\documentclass[11pt,leqno,fleqn]{sn-jnl} 

\usepackage{amsmath,amssymb,amsfonts}
\usepackage{amsthm}
\usepackage{xcolor}
\newtheorem{theorem}{Theorem} 
\newtheorem{proposition}[theorem]{Proposition}%

\newtheorem{example}[theorem]{Example}%
\newtheorem{lemma}[theorem]{Lemma}%

\theoremstyle{thmstylethree}%
\newcommand{\trans}{^\top}
\newcommand{\ba}{\mathbf a}
\newcommand{\bb}{\mathbf b}
\newcommand{\be}{\mathbf e}
\newcommand{\bff}{\mathbf f}
\newcommand{\bw}{\mathbf w}
\newcommand{\bx}{\mathbf x}
\newcommand{\by}{\mathbf y}
\newcommand{\bz}{\mathbf z}
\newcommand{\bmu}{\boldsymbol{\mu}}
\newcommand{\bxi}{\boldsymbol{\xi}}
\newcommand{\zero}{\mathbf{0}}
\newcommand{\calf}{\mathcal F}
\newcommand{\caln}{\mathcal N}
\newcommand{\calo}{\mathcal O}
\newcommand{\calu}{\mathcal U}
\newcommand{\calv}{\mathcal V}
\newcommand{\calz}{\mathcal Z}
\newcommand{\R}{\mathbb R}
\newcommand{\C}{\mathbb C}
\newcommand{\wh}{\widehat}
\newcommand{\wt}{\widetilde}
\newcommand{\eps}{\varepsilon}

\newcommand{\ex}{\mathbb{E}} 
\newcommand{\Afix}{A^{\#}}
\newcommand{\tr}{{\rm tr}}
\newcommand{\fixp}{\bx_{\infty}}

\newcommand{\mh}[1]{#1}

\begin{document}

\title
{On spectral properties and fast initial convergence \\ of the Kaczmarz method}

\author[1]{\fnm{Per Christian} \sur{Hansen}}\email{pcha@dtu.dk}

\author[2]{\fnm{Michiel E.} \sur{Hochstenbach}}\email{m.e.hochstenbach@tue.nl}

\affil[1]{\orgdiv{Department of Applied Mathematics and Computer Science},
  \orgname{Technical University of Denmark}, \orgaddress{
  \city{Kgs.~Lyngby}, \postcode{DK-2800}, \country{Denmark}}}

\affil[2]{\orgdiv{Department of Mathematics and Computer Science},
  \orgname{TU Eindhoven}, \orgaddress{\street{PO Box 513}, \city{Eindhoven},
  \postcode{5600 MB}, \country{The Netherlands}}}

\abstract{The Kaczmarz method is successfully used for solving
discretizations of linear inverse problems, especially in computed tomography
where it is known as ART.
Practitioners often observe and appreciate its fast convergence in the first few iterations,
leading to the same favorable semi-convergence that we observe for
simultaneous iterative reconstruction methods.
While the latter methods have symmetric and positive definite iteration operators that facilitate their analysis,
the operator in Kaczmarz's method is nonsymmetric and it has been an open question so far to understand this fast initial convergence.
We perform a spectral analysis of Kaczmarz's method that gives new insight into its (often fast) initial behavior.
We also carry out a statistical analysis of how the data noise enters the iteration vectors,
which sheds new light on the semi-convergence.
Our results are illustrated with several numerical examples.}

\keywords{Kaczmarz, spectral properties, initial behavior, semi-convergence, asymptotic convergence, statistical analysis, symmetric Kaczmarz}

\pacs[MSC Classification]{65F10, 65F22, 65R30, 65R32, 65F50}

\maketitle
\begin{center}
{\em Dedicated with pleasure to our precious friend {\AA}ke Bj\"orck, \\ who is an inspiration and encouragement to us}
\end{center}

\section{Introduction}  
We consider a linear discrete ill-posed problem
\begin{equation} \label{Ax=b}
A\, \bx = \bb, \qquad A \in \R^{m \times n}
\end{equation}
obtained from discretization of an inverse problem, and
we may have any of the cases $m=n$, $m>n$, or $m<n$.
For such systems, the singular values of $A$ decay gradually to zero and the coefficients of $\bb$,
in the basis of the singular vectors, decay faster than the singular values.
We assume that \eqref{Ax=b} is consistent.
In addition, we suppose the very natural and standard condition that $A$ does not have zero rows (these can just be omitted from the system).
For details about discretizations of inverse problems see, e.g., \cite{HansenDIP}.
As we will also discuss in Section~\ref{sec:stage}, we are interested in the least-norm solution: we assume that $\bx$ is orthogonal to the nullspace of $A$, as it is not possible to reconstruct a component in the nullspace without additional information.

For large-scale problems, such as those that arise in X-ray computed tomography (CT) \cite{Beister}, we need to solve \eqref{Ax=b} by means of iterative methods and
the behavior of many of these methods, when applied to inverse problems, are well understood.
In this paper we consider a method that still needs scrutinization, namely,
the Kaczmarz method \cite{Kaczmarz} (see also \cite{GBH70, Tanabe, Cen81, EN08, Elfving18})
that repeatedly cycles through the rows of $A$.
Variants of Kaczmarz's method that access the rows in random order (see, e.g., \cite{FAM24} for a recent review) are not our focus, although we compare a deterministic and randomized Kaczmarz method in an experiment.

The Kaczmarz method may be seen as a fixed-point process,
and for a zero initial guess it converges to the unique least-norm solution; see Section~\ref{sec:stage} for more details.
Since we know that the method converges, the spectral radius of the corresponding iteration operator should be less than 1.
In fact, it may be very close to 1, resulting in a extremely slow asymptotic convergence behavior.
Yet, for many discrete ill-posed problems the Kaczmarz iterations make a very significant progress in the first few iterations, and the reason for this has been an open problem thus far.
The main goal of our paper is to study and explain this initial behavior.
We also discuss the so-called symmetric Kaczmarz method and its relations with the standard approach.

For the inverse problems that we consider there usually is noise in the right-hand side~$\bb$.
In this case we observe {\em semi-convergence}: the iterates initially get closer to the desired but unavailable solution of
the noise-free system, but then diverge because of the influence of the noise.
A second goal of our paper is to explain this behavior with statistical insight inspired by \cite{HansenLAA}.

In the Kaczmarz method and many other iterative methods that exhibit this semi-convergence,
the number of iterations $k$ may play the role of a regularization parameter, and we want to
stop the iterations right before the influence of noise starts to dominate.
Such stopping rules are not a topic of this paper;
they can be based on the same principles \cite{RR13} that underlie conventional
regularization methods such as Tikhonov regularization, e.g., by stopping when the
residual norm is of the same size of the norm of the errors in~$\bb$;
see \cite{HaJR21} for an overview of such stopping rules.

Our paper is organized as follows.
In Section~\ref{sec:stage}, we summarize the Kaczmarz method and introduce important notation.
We consider the convergence for the case of
noise-free data in Section~\ref{sec:eigen}; specifically, we study the eigenvalues of the iteration operator to explain the
convergence of Kaczmarz's method and obtain insight into the fast initial decay of the iteration error.
Section~\ref{sec:stat} then studies how the convergence is influenced by noise in the data;
this provides insight into the semi-convergence behavior.
Throughout the paper, we illustrate our theory and insight with numerical examples;
the main application of Kaczmarz's method is CT which we illustrate with
test problems from {\sc AIR Tools II} \cite{AIRtoolsII};
we also use a few simpler test problems from {\sc Regularization Tools} \cite{RegTools}.

We use the following notations.
We write $\| \cdot \|$ for the vector and matrix 2-norm,
$\| \cdot \|_{\rm F}$ for the Frobenius norm,
$\Box\trans, \, \Box^*$ for the transpose and complex conjugate of $\Box$,
and $\ex(\cdot)$ for expected value.
Moreover, $\ba_i\trans$ is the transpose of the $i$th row of matrix $A$,
$\be_i$ is the $i$th unit vector,
$I$ is the identity matrix of appropriate dimension,
$\bx_k$ is the $k$th iteration vector,
$\zero$ is the zero vector of appropriate dimension,
$\sigma$ is the standard deviation of the noise, and
$\sigma_{\min}(A)$ is the smallest singular value of $A$.
Finally,
$\rho(\cdot)$ is the spectral radius of a matrix (the largest eigenvalue in absolute value),
${\rm spec}(\cdot)$ denotes the spectrum of a square matrix,
and ${\rm diag}(\cdot)$ is a diagonal matrix whose elements are given by the argument.

\section{Setting the stage: Kaczmarz's method} \label{sec:stage} 

We start by recalling the Kaczmarz method and the related symmetric Kaczmarz method.
Given an initial vector $\bx_0$ (where the standard choice is $\bx_0 = \zero$ as we will discuss shortly), Kaczmarz's method performs cycles (or sweeps) of steps using the rows $\ba_i\trans$ of $A$.
This means that for $k=0, 1, \dots$ we carry out:
  \begin{align*}
    \bx_{k+1}^{(0)}   & = \bx_k, \\
    \bx_{k+1}^{(i+1)} & = \bx_{k+1}^{(i)} + \omega \ \frac{\bb-\ba_i\trans \bx_{k+1}^{(i)} }{\|\ba_i\|^2} \ \ba_i,
      \qquad i = 1, \dots, m, \\
    \bx_{k+1}         & = \bx_{k+1}^{(m)}
  \end{align*}
where $\omega$ is a relaxation parameter with $0 < \omega < 2$.
We refer to $\bx_k$ as the $k$th iteration vector.

Now decompose $AA\trans = \wh L + D + \wh L\trans$,
where $\wh L$ is the strictly lower triangular part and $D$ is the diagonal part of $A\, A\trans$.
One sweep of Kaczmarz's method can then be written as \cite{ENi09} (see also \cite{Bart})
  \begin{equation} \label{kaczmarz}
    \bx_{k+1} = \bx_k + A\trans L^{-1} \, (\bb-A\, \bx_k) \ , \qquad
    L = L_{\omega} = \wh L + \omega^{-1} \, D \ .
  \end{equation}
Here $L$ is nonsingular, since $A$ does not have any zero rows.
We can also write the iterative process as
\begin{equation} \label{Git}
\bx_{k+1} = G \ \bx_k + A\trans \, L^{-1} \, \bb
\end{equation}
with iteration matrix
  \[
    G = G_{\omega} := I - A\trans L^{-1} A \ .
  \]
Let $A = U \Sigma V\trans$ be the SVD of $A$, where $U \in \R^{m \times r}$, $\Sigma \in \R^{r \times r}$,
and $V \in \R^{n \times r}$, with $r = {\rm rank}(A)$.
Denote $\calv := {\rm range}(V) = {\rm range}(A\trans) = {\rm null}(A)^{\perp}$, and similarly $\calu := {\rm range}(U) = {\rm range}(A)$.
If $A$ is of full column rank, then $\calv = \R^n$ is the entire space.

It is important to notice that any component of $\bx_0$ that is in $\calv^{\perp} = {\rm null}(A)$ is not annihilated by multiplication with~$G$.
Conversely, if $\bx_0 \in \calv$, it follows that $\bx_k \in \calv$ for all $k$, as ${\rm range}(A\trans) = \calv$.
For this reason, as well as for the quality of the final solution, $\bx_0 = \zero$ is commonly chosen as a starting vector, which we will assume from now on.
Then the interesting action of $G$ takes place on $\calv$, and it suffices to consider the restriction $G\vert_{\calv}$ of $G$ to $\calv$.
The resulting operators $A\trans L^{-1} A\vert_{\calv}$ and $G\vert_{\calv}$ are maps from $\calv$ to $\calv$.

The map $A\trans L^{-1} A$ is invertible on $\calv$, and $A\trans L^{-1} A\vert_{\calv}$ has an inverse $(A\trans L^{-1} A\vert_{\calv})^{-1}$.
We note that $(A\trans L^{-1} A\vert_{\calv})^{-1}$ can be extended to the entire space $\R^n$ by defining it to be zero on $\calv^{\perp}$; this gives the pseudoinverse $(A\trans L^{-1} A)^+$.
However, the action on $\calv$ is what matters, and we do not need to consider this pseudoinverse.

It is well known that the Kaczmarz method converges for consistent linear systems \eqref{Ax=b} and $0 < \omega < 2$ (see, e.g., \cite{Pop18}); this is equivalent to $\rho(G\vert_{\calv}) < 1$.
We denote the fixed point of \eqref{Git} by $\fixp \in \calv$.
It is easy to see that it satisfies
\begin{equation} \label{ls2}
(A\trans L^{-1} \, A\vert_{\calv}) \ \fixp = A\trans L^{-1} \, \bb \, .
\end{equation}
It is shown in \cite[Lemma~2.2]{Bart} that for consistent systems, this is equivalent to \eqref{Ax=b}.
Therefore, the Kaczmarz method converges to the unique least-norm solution of \eqref{Ax=b}, which is in $\calv$ and orthogonal to the nullspace $\calv^{\perp}$:
  \begin{equation} \label{eq:fixpoint}
    \fixp = \Afix \, \bb \qquad \hbox{with} \qquad
    \Afix = (A\trans L^{-1} \, A\vert_{\calv})^{-1} \, A\trans L^{-1} \ .
  \end{equation}
In particular, we point out that although the iterations depend on $\omega$, the solution $\fixp$ is independent of $\omega$.
However, the speed of convergence generally depends on this parameter.
Elfving and Nikazad \cite[p.~5]{ENi09} remark on \eqref{ls2}
\begin{quote}
``these equations do not correspond to a gradient mapping. It follows that there is no underlying function which is minimized.''
\end{quote}
In other words, since $L^{-1}$ is not symmetric positive definite, these are not normal equations corresponding to a weighted least squares approach.

We can rewrite \eqref{kaczmarz} in the form
$\bx_{k+1}-\fixp = (I-A\trans L^{-1}A) \, (\bx_k-\fixp)$.
In terms of the error vector $\bff_k := \bx_k-\fixp$ this means
\begin{equation} \label{err}
\bff_{k+1} = (I-A\trans L^{-1}A) \ \bff_k.
\end{equation}
A necessary condition for convergence to the minimum-norm solution is that $\bff_0 \in \calv$;
since $\fixp \in \calv$, this is equivalent to $\bx_0 \in \calv$.

As a side note, we point out that for (very) ill-posed problems, rank decisions may be nontrivial when there are one of more tiny singular values, especially in the situation without a clear {\em gap} between small singular values.
Although for these problems determining the spaces $\calv$ and $\calu$ may not be well-posed, this has very little influence on the results of this paper (e.g., it may very slightly change the spectral radius $\rho(G\vert_{\calv})$ in experiments in Section~\ref{sec:eigen}).

There is a lesser known and used version, the {\em symmetric Kaczmarz} method,
in which one cycle consists of a downward sweep of the rows, followed by an upward sweep.
This approach has been proposed (without this name) in \cite[(4.1)]{BjEl79}; see also \cite{EN08}.
In the above framework we access the rows of $A$ in the order
$i=1,2,\ldots,m, \, m,m-1,\ldots,1$;
the double steps for $i=1$ and $i=m$ are redundant when $\omega=1$.
In \cite[Prop.~2.6]{Bart} it has been shown that an upsweep of the symmetric Kaczmarz method corresponds to the transpose operator
  \begin{equation} \label{upG}
    G\trans = I - A\trans \, L^{-\top} A
  \end{equation}
and therefore the symmetric Kaczmarz method has the iteration matrix
  \begin{align*}
    G_{\rm s} & = G_{\rm s}(\omega) := G\trans\, G = (I - A\trans \, L^{-\top} A) \, (I - A\trans \, L^{-1} A) \\
    & = I - A\trans \, L^{-1} A - A\trans \, L^{-\top} A + A\trans \, L^{-\top} \, (\wh L + \omega^{-1} \,D + \wh L\trans + \omega^{-1}\,D + (1-2\,\omega^{-1}) \, D) \, L^{-1} \, A \\
    & = I - (2\,\omega^{-1}-1) \, A\trans \, L^{-\top} \, D \, L^{-1} \, A.
  \end{align*}
Similar to \eqref{kaczmarz} we can therefore write the symmetric Kaczmarz iterations as
  \[
    \bx_{k+1} = \bx_k + A\trans S \ (\bb-A\, \bx_k)
  \]
where $S = (2\omega^{-1}-1) \, L^{-\top} \, D \, L^{-1}$;
see also \cite[Prop.~3.3]{EN08}.
The matrix $S$ is symmetric and positive definite, and hence the symmetric Kaczmarz method
is a member of the family of simultaneous iterative reconstruction technique (SIRT) type methods (see, e.g., \cite{ENH10}).
Just as the standard method, the symmetric Kaczmarz method converges for $0 < \omega < 2$ and $\bx_0 = \zero$ to the unique minimum-norm
solution, which in this symmetric case minimizes $\| S^{1/2} (A\,\bx-\bb) \|$ (cf.~\cite[Thm.~1.1]{ENH10}).
We are going to analyze the standard Kaczmarz method together with the symmetric variant in the next section.

\section{Eigenvalues of the iteration operator} \label{sec:eigen} 

In this section we study the eigenvalues of the iteration operator and their influence on the initial convergence.
We usually speak of ``operator'' rather than matrix, as it is the restriction $G\vert_{\calv}$ what is relevant.
A difference between $G$ and $G\vert_{\calv}$ is that the former may have extra eigenvalues equal to 1, which are irrelevant for the spectral radius $\rho(G\vert_{\calv})$.
We focus on noise-free data to obtain insight into the
iteration error, and we leave the study of the influence of data noise to Section~\ref{sec:stat}.
We present new results for several aspects of the convergence of the Kaczmarz method and illustrate them with examples.

\subsection[ ]{Spectral properties of $G$: introduction}
First, note that we can write
  \[
    \bx_k = A_k^{\#} \,  \bb
  \]
in which
  \begin{equation} \label{Ak}
    A_k^{\#} = (I + G + \cdots + G^{k-1}) \, A\trans L^{-1}
    = (I-G^k)\, (I-G)^{-1} \, A\trans L^{-1} = (I-G^k)\, \Afix \ ,
  \end{equation}
where $\Afix$ defines the fixed point; cf.~\eqref{eq:fixpoint}.
We make the (very mild) assumption that $G$ is diagonalizable, with eigenvalue decomposition
  \begin{equation} \label{eigG}
    G\vert_{\calv} = W\, \Lambda \, W^{-1}, \qquad \Lambda = {\rm diag}(\lambda_i) \ .
  \end{equation}
The restriction means here that we only consider $\calv$ as domain of $G$; the range is also subset of $\calv$.
This means that we compute the eigenvalue decomposition $V\trans GV = C \Lambda \, C^{-1}$, and then take $W := VC$.
Since $G$ is nonsymmetric, some eigenvalues and associated eigenvectors may be complex.
The eigenvalues have magnitudes less than 1, and hence
$A_k^{\#} \rightarrow \Afix$ for $k \rightarrow \infty$.

We denote the open unit disk by ${\cal Z} = \{ \, |z| < 1 \, : \, z \in \mathbb C \, \}$.
Since Kaczmarz's method converges, the eigenvalues of the operator $G\vert_{\calv}$ are inside $\calz$.
Typically, $G\vert_{\calv}$ has several (real) eigenvalues very close to 1, which makes the asymptotic convergence very slow.
In addition, $G\vert_{\calv}$ often has one or more eigenvalues (very) close to 0, which is the topic of Section~\ref{sec:zero}.
Then $A\trans L^{-1} A\vert_{\calv} = (I-G)\vert_{\calv}$ has the eigenvalue decomposition
  \begin{equation} \label{eq:ED}
    A\trans L^{-1} A\vert_{\calv} = W\, (I-\Lambda) \, W^{-1}.
  \end{equation}
As Kaczmarz's method converges for $0 < \omega < 2$, this implies $\rho(G\vert_{\calv}) < 1$.
However, $\rho(G\vert_{\calv})$ may be extremely close to 1; for very ill-conditioned $A$, it may even be numerically equal to 1, i.e., different from 1 by machine precision.
This explains the flat plateau for the iteration error that we regularly observe for Kaczmarz's method.

We first start with the following lemma, based on the well-known result that nonzero eigenvalues of $BC$ are equal to those of $CB$ for matrices of appropriate dimension (see, e.g., \cite{HJo85}).

\begin{lemma} \label{lem:eigAB}
For any $A \in \R^{m \times n}$ and nonsingular $L \in \R^{m \times m}$,
\[
{\rm spec}(A\trans L^{-1} A\vert_\calv)
= {\rm spec}(L^{-1} AA\trans \vert_\calu)
= {\rm spec}(A A\trans L^{-1}\vert_{L\calu}).
\]
\end{lemma}
\vspace{-5mm}
\begin{proof}
The nonzero eigenvalues of $BC$ and of $CB$ are identical; the restrictions are such that exactly the nonzero eigenvalues of the operators are selected.
\end{proof}
In addition, there is also a relation between the eigenvectors of the operators
in Lemma \ref{lem:eigAB}, which we will comment on in Section~\ref{sec:zero}.

Assuming a zero initial guess, a necessary and sufficient condition
for convergence to $\fixp$ is $\rho((A\trans L^{-1} A-I_n)\vert_\calv) < 1$.
This is the topic of the next proposition.
By $\calz+1$ we mean the shifted set $\{\,z+1 \,:\, z \in \calz\,\}$; this is the open disk with center and radius equal to 1.
For clarity, we add a subscript for the dimension of the identities in the proposition ($I_m$ and $I_n$).

\begin{figure}[htb!]
\centering
\includegraphics[width=0.50\linewidth]{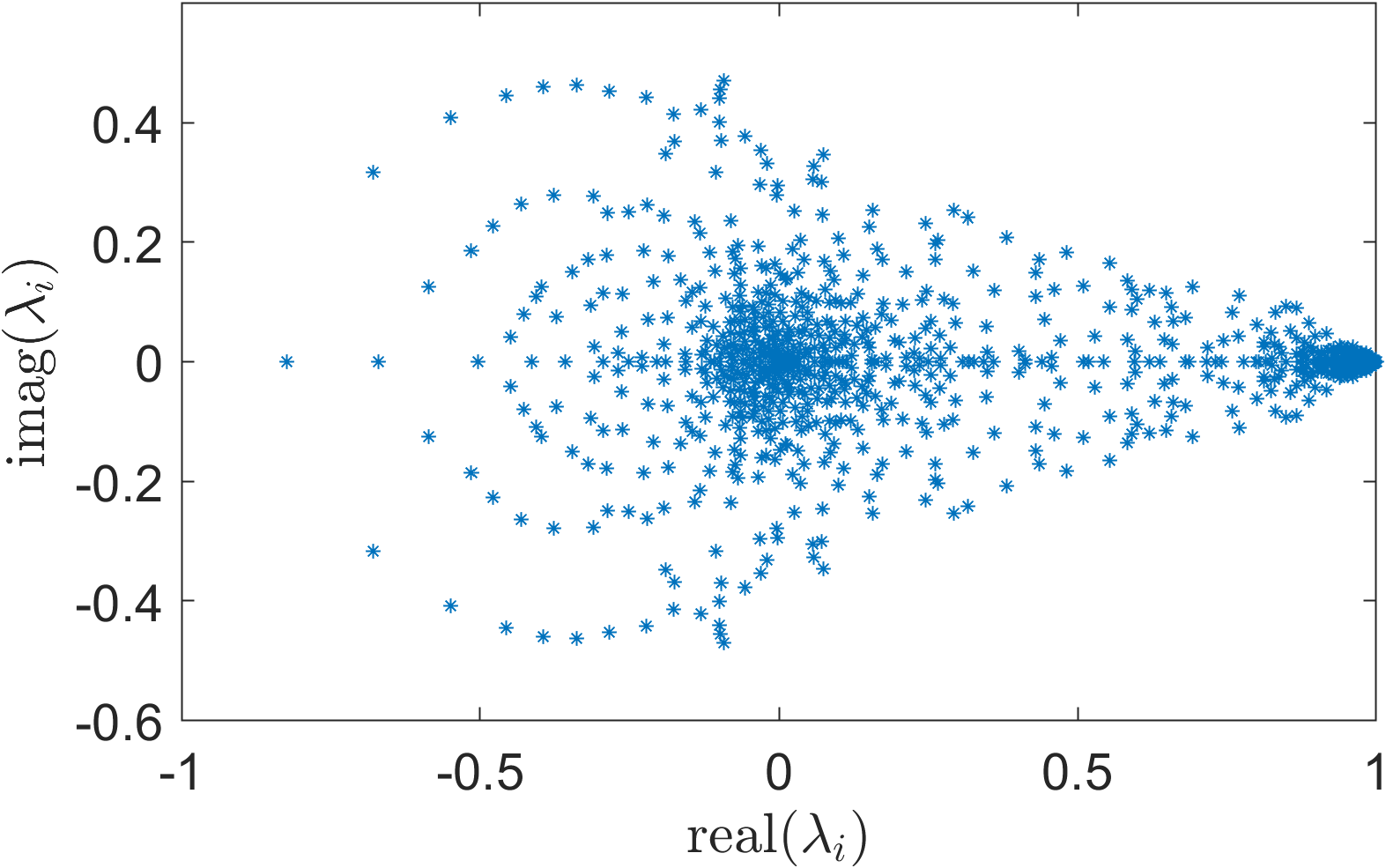}
\caption{A typical distribution of the eigenvalues of $G$ inside the complex unit disk for an X-ray CT problem.
The spectral radius here is $\rho(G) = 0.9999998937$.}
\label{fig:Evalplot}
\end{figure}

\begin{proposition} \label{prop:equiv}
Let $L \in \R^{m \times m}$ be nonsingular. Iteration \eqref{kaczmarz} converges to $\fixp$ for any $\bx_0 \in \calv^{\perp}$ if and only if one of the following equivalent conditions holds:
\begin{description} \vspace{-1mm} \itemsep=1mm
\item[(a)] $\rho((A\trans L^{-1} A-I_n)\vert_\calv) < 1$;
\item[(b)] $\rho((L^{-1}A A\trans-I_m)\vert_\calu) < 1$;
\item[(c)] $\rho((A A\trans L^{-1}-I_m)\vert_{L\calu}) < 1$;
\item[(d)] $\rho((AA\trans- L)\vert_\calu, \  L\vert_\calu) < 1$;
\item[(e)] ${\rm spec}(AA\trans\vert_\calu, \  L\vert_\calu) \subseteq {\cal Z}+1$.
\end{description}
\end{proposition}
\vspace{-5mm}
\begin{proof}
These properties follow directly from Lemma~\ref{lem:eigAB}; the identity only
performs a shift.
\end{proof}
\noindent

Figure~\ref{fig:Evalplot}, illustrating Proposition~\ref{prop:equiv}, shows a typical distribution of the eigenvalues of $G$; this plot is for an X-ray CT problem.
We will consider spectra like these in the next subsection.

\subsection[ ]{Rapid initial convergence and zero eigenvalues of $G$} \label{sec:zero}
This subsection is central in understanding the spectral properties investigated in Section~\ref{sec:eigen}.
We now address the (long-time) open problem of explaining the fast initial convergence of Kaczmarz's method, as is often observed and appreciated by practitioners.
We point out in this section that (near) zero eigenvalues of $G$ are key to understand this behavior, and that we have a simple expression for one of more corresponding eigenvectors for $\omega=1$.
We explain why the zero eigenvalue may have a multiplicity larger than 1, and why there may additionally be near-zero eigenvalues, also for $\omega \ne 1$.

First, we point out that $G$ has a zero eigenvalue for $\omega = 1$.
Denote $L_1$ for the case $\omega=1$ in $L = L_{\omega}$.

\begin{proposition} \label{prop:zero}
Let $\omega = 1$.
Then $G = I-A\trans L_1^{-1} A$ has a zero eigenvalue with corresponding eigenvector $\ba_1 = A\trans \be_1$, and left eigenvector $\ba_n = A\trans\be_n$.
\end{proposition}
\vspace{-6mm}
\begin{proof}
Since $AA\trans \be_1 = L_1 \, \be_1$, we have
  \[
    A\trans L_1^{-1} \, A \, (A\trans \be_1) = A\trans L_1^{-1} \, (L_1 \, \be_1)  = A\trans \be_1 \ .
  \]
Similarly, $AA\trans \be_n = L_1\trans \be_n$, which means that
$A\trans L_1^{-\top} \, A \, (A\trans \be_n) = A\trans \be_n$.
\end{proof}

In other words, the first row $\ba_1$ is a zero eigenvector of the downward sweep operator $G$, and the last row $\ba_n$ an eigenvector of the upsweep operator $G\trans$.
As a side note, related to the sentence after Lemma~\ref{lem:eigAB}, we have that $\be_1$ is an eigenvector of $L_1^{-1} AA\trans$ and the pencil $(AA\trans, L_1)$ corresponding to eigenvalue $1$;
$L_1\be_1$ is the associated eigenvector of $AA\trans L_1^{-1}$.

A zero eigenvalue of $G$ means that the corresponding mode converges after one iteration (sweep):
it is a direction in which the iterations $\bx_k$ converge to the fixed point $\fixp = (A\trans L^{-1} \, A)^{-1} \, A\trans \bb$ in a single step, as follows.
Considering \eqref{eigG}, let (without loss of generality) $\bw_1 = \ba_1$ denote the eigenvector corresponding to eigenvalue 0 of $G$.
Decompose the initial error $\bff_0 = \bx_0-\fixp = \sum_i \gamma_i \, \bw_i$.
Then the error component of $\bff_1 = G\,\bff_0$ in the direction of $\bw_1$ is zero; it is annihilated after the very first step.

We can make the following first key observation.
\begin{quote}
\fbox{\parbox[t]{15cm}{
The Kaczmarz method with (default value) $\omega=1$ makes
rapid immediate progress if the fixed point $\fixp$ has a considerable component in the direction of $\ba_1$.
Therefore, the speed of the initial convergence depends particularly on $\fixp$ and $\ba_1$, and hence the ordering of the rows.
}}
\end{quote}

For some problems, $G$ may have not only one, but several zero eigenvalues for $\omega=1$, particularly for CT applications.
We now explain the reason for this and illustrate it with an example.

Since $L_1$ is the lower triangular part of $AA\trans$, it is easy to see that for $2 \le j \le n$
\begin{equation} \label{extra-terms}
AA\trans \be_j = L_1\,\be_j + \sum_{i<j} \, (\ba_i\trans\ba_j) \ \be_j.
\end{equation}
For some problems, especially from CT type of problems, several rows of $A$ are {\em structurally orthogonal}, i.e., they have mutual disjoint sets of indices of nonzero elements.
This means that for several (or even many) rows $\ba_i$ and $\ba_j$ we have $\ba_i\trans\ba_j = 0$ (even in finite precision arithmetic, this quantity is {\em exactly} zero).
This means that $\ba_2$ is an eigenvector corresponding to eigenvalue zero when $\ba_1^T\ba_2 = 0$;
furthermore, $\ba_3$ is an eigenvector corresponding to eigenvalue zero when $\ba_1^T\ba_3 = \ba_2^T\ba_3 = 0$, etc.

\begin{example} \rm
We illustrate this, and the influence of row orderings of $A$ with an example from X-ray CT, which is available as
the test problem \textsf{paralleltomo} from {\sc AIR Tools II} \cite{AIRtoolsII}.
We use a $32 \times 32$ image, 32 projection angles, and $32$ X-rays per angle;
this gives a square matrix $A$ with $m=n=1024$ and it has full rank.
The corresponding matrix $G$ has several (numerically) zero eigenvalues.
For the default row ordering the initial principal submatrix of $AA\trans$ is sparse; this is caused by many rows being
structurally orthogonal (due to the positions of the nonzeros).
In Figure~\ref{fig:influence} we see that the $20 \times 20$
leading submatrix is even diagonal (blue sparsity pattern plot).
In view of \eqref{extra-terms}, this explains the many zero eigenvalues (blue graph displaying $|\lambda_i|$).
We also form a matrix $\wt A$ by a random permutation of the rows of $A$.
The matrix $\wt AA\trans$ (red plot) is still a bit sparse, but less so, and this has the effect of still many but fewer (numerically) zero eigenvalues.
Indeed, the convergence in the first few steps with both row orderings is very rapid.

\begin{figure}[htb!]
\centering
\includegraphics[width=0.50\linewidth]{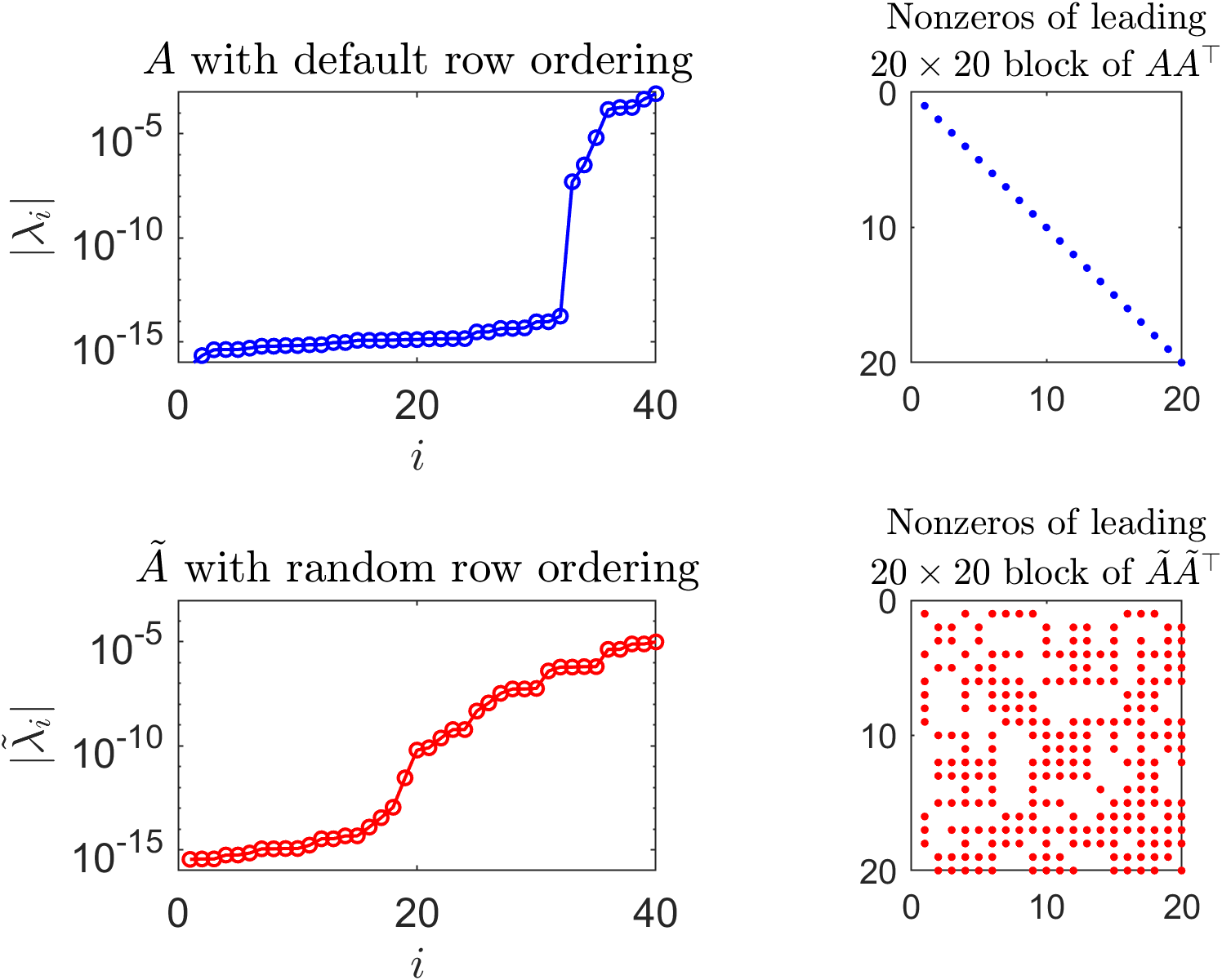} \quad
\includegraphics[width=0.45\linewidth]{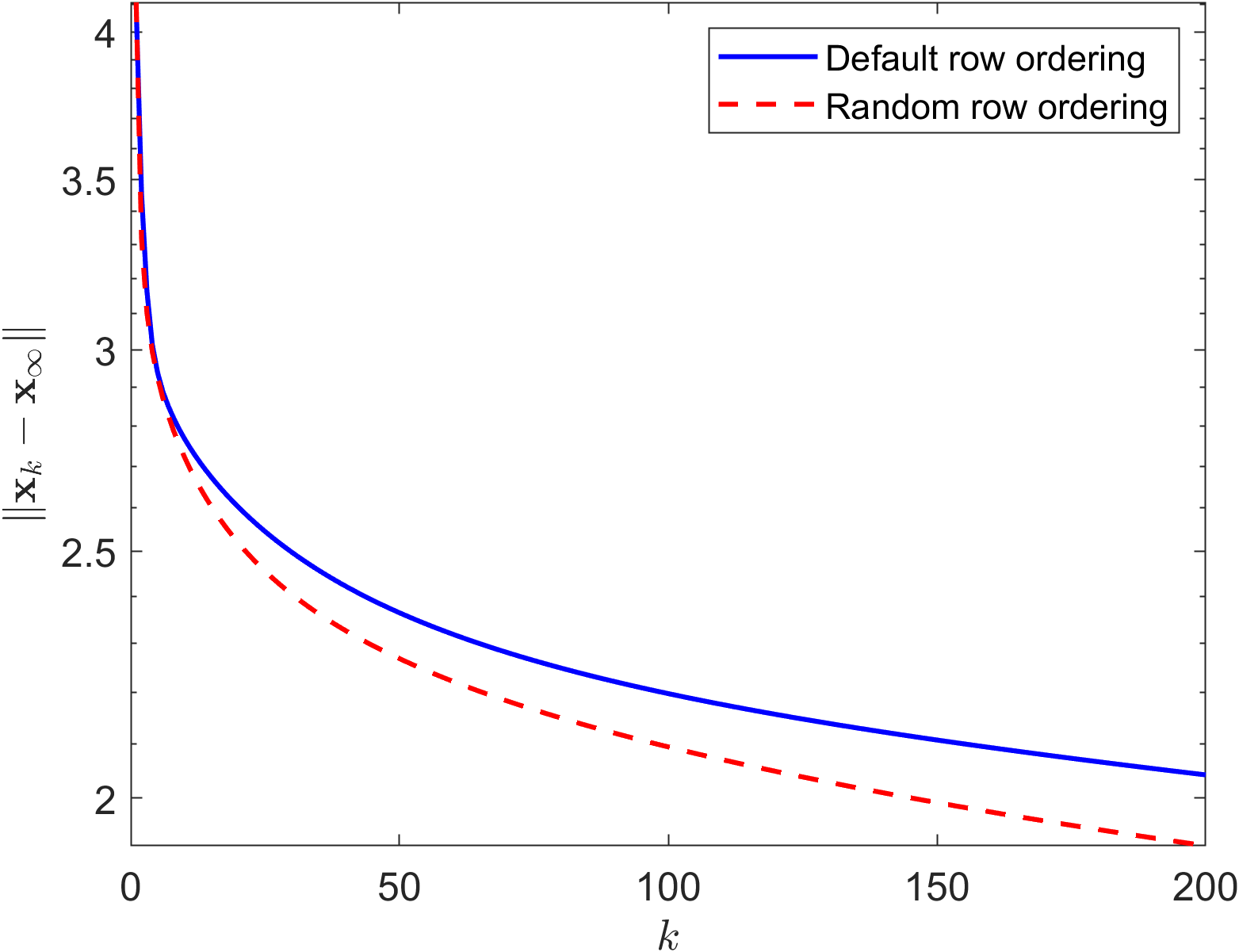}
\caption{CT test problem (see text for details).
Left:\ the small eigenvalues of $G$ and $\wt G$ corresponding, respectively, to
the matrices $A$ and $\wt A$ with default and random row ordering.
Middle:\ the nonzeros of the leading submatrices of $AA\trans$ and $\wt A\wt A\trans$.
Right:\ the error histories for the two row orderings.}
\label{fig:influence}
\end{figure}

The error histories for noise-free data are shown in the same figure.
The zero eigenvalues are of big influence for the initial convergence, but less so for the asymptotic convergence, where the spectral radius is important.
A permutation generally also changes this quantity:
$\rho(G) \approx 1-1.1 \cdot 10^{-7}$ and $\rho(\wt G) \approx 1-1.1 \cdot 10^{-6}$.
Indeed, the (slightly) smaller spectral radius of the randomized order $\wt G$ is slightly faster after 200 iterations.
\end{example}

In addition, several deterministic row orderings have been proposed such that  subsequent rows of $A$ are {\em nearly} orthogonal, as this is usually favorable for successive projection methods such as Kaczmarz.
Examples can be found in \cite{Pop99, Gor06} and \cite[App.~A]{SHa14}.
Effects of several types of randomized row orderings has been reviewed in \cite{FAM24}.

Also when the first rows are not structurally orthogonal, $G$ may have near-zero eigenvalues, which we will explain now.
The first row $\ba_1$ is always a ``zero eigenvector'';
the second row when $\ba_1^T\ba_2 = 0$.
When $|\ba_1\trans\ba_2|$ is nonzero but small (which may often be the case for ``good orderings'' in the Kaczmarz method), then $\ba_2$ is an {\em approximate} eigenvector of $G$ corresponding to an eigenvalue close to 0, because of the following.
Under the very mild assumption that $L_1^{-1} AA\trans$ is diagonalizable, the Bauer--Fike Theorem (see, e.g., \cite[Thm.~3.6]{Saa11}) guarantees the existence of an eigenvalue $\lambda$ of $G$ with
\[
|\lambda| \le \kappa(X) \cdot |\ba_1\trans\ba_2| \cdot \|L_1^{-1} \be_1\|,
\]
where $\kappa(X)$ is the condition number of the eigenvector matrix $X$ of $L_1^{-1} AA\trans$.
Therefore, apart from the zero eigenvalue corresponding to eigenvector $\ba_1$, there is another eigenvalue with upper bound on the absolute value which scales linearly with $|\ba_1\trans\ba_2|$.
Good row orderings may therefore have favorable effects, not only for the asymptotic convergence (smaller $\rho(G\vert_{\calv})$), but also the initial stage (near-zero eigenvalues).

Next, we turn our attention to the situation $\omega \ne 1$. In this case it is generally not guaranteed that there is a zero eigenvalue of $G$.
However, there are several examples where we have one or more zero eigenvalues for a wide range of $\omega$ around 1; see Figure~\ref{fig:many}.

\begin{figure}[htb!]
\centering
\includegraphics[width=0.4\linewidth]{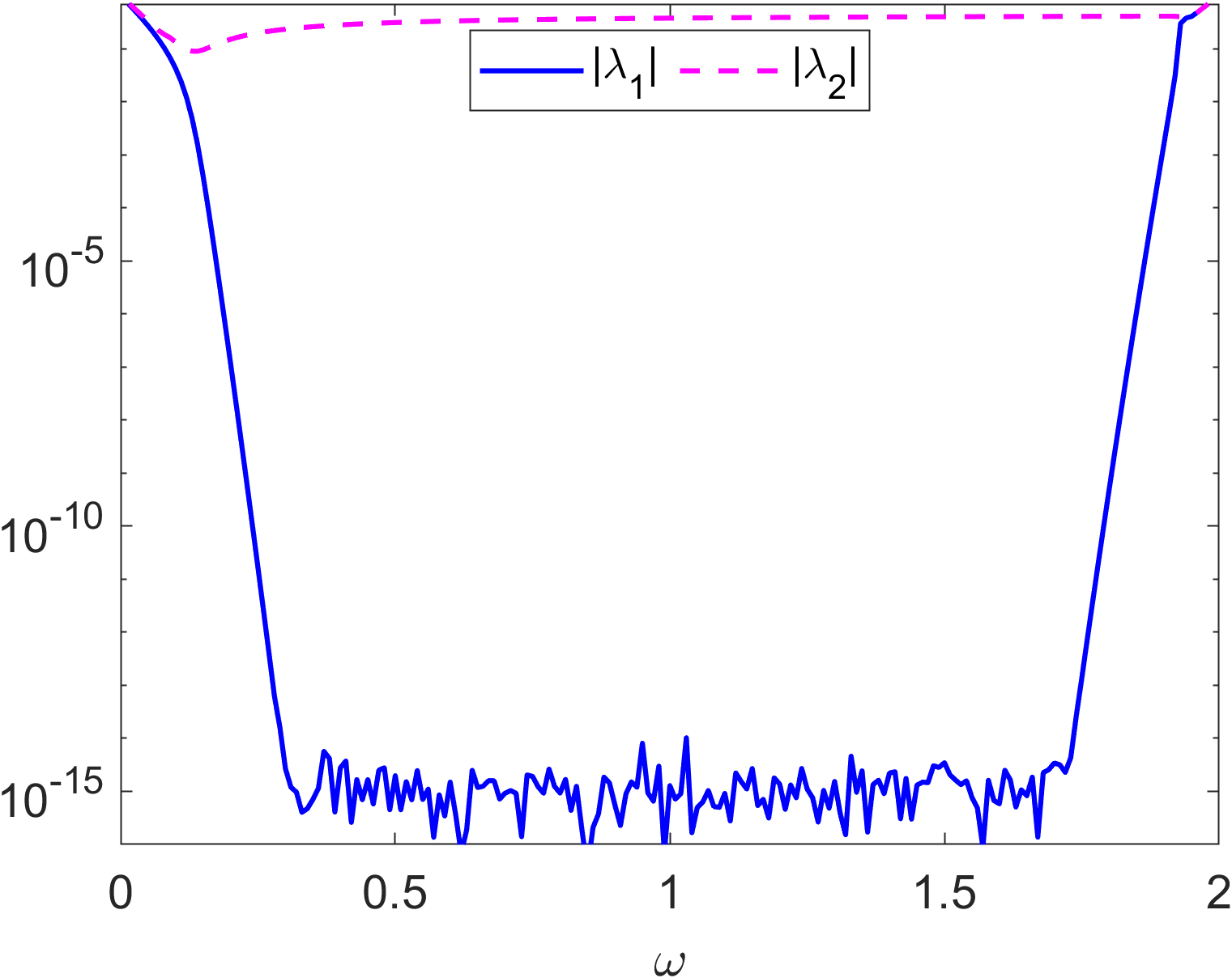} \quad
\includegraphics[width=0.4\linewidth]{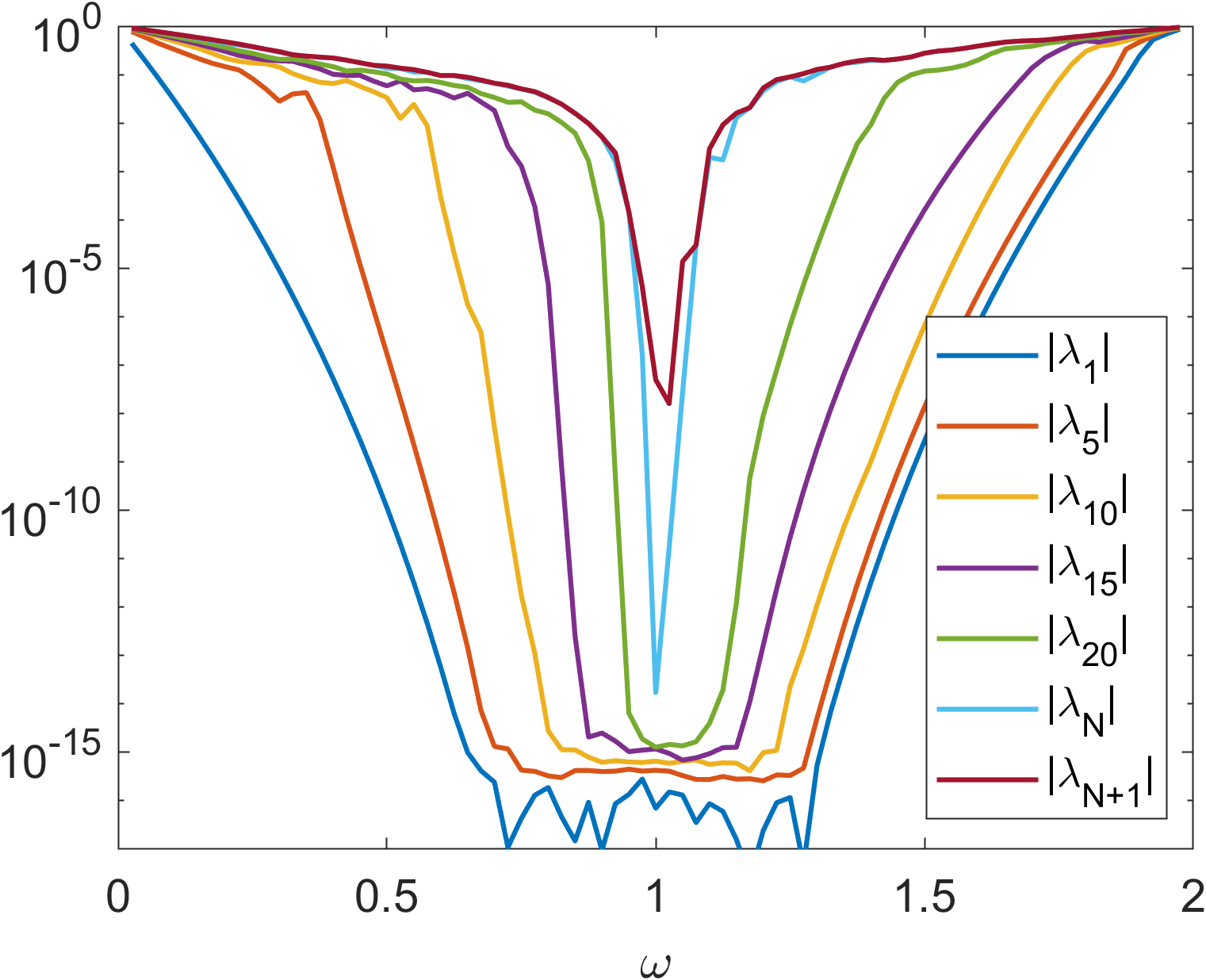}
\caption{Problems {\sf gravity} (${\mathtt d} = 0.01$ and $n=128$) and {\sf paralleltomo}
(the image size is $32\times 32$, and $A$ is $1024 \times 1024$), with, respectively, one zero eigenvalue, and several zero eigenvalues for an interval around $\omega=1$.}
\label{fig:many}
\end{figure}

The left figure is obtained from {\sf gravity}, the 1D gravity surveying problem from \cite{RegTools}.
This problem includes a parameter \texttt{d} which can be used to control the condition number of $A$.
The matrix is $128 \times 128$ and the singular values decay approximately as $e^{-0.7i}$.
The resulting $G$ has a numerically zero eigenvalue for a wide range of $\omega$-values.
The default row ordering in {\sf paralleltomo}, the right graph, is such that the leading $30 \times 30$ submatrix of $AA\trans$ is diagonal, meaning that many of the initial rows are structurally orthogonal.
This results in a favorable situation (cf.~\eqref{extra-terms}): a zero eigenvalue of high multiplicity for $\omega=1$, with also similar behavior for neighboring $\omega$-values.
The reason for the large plateau of the zero eigenvalue for {\sf gravity} (for $\omega \in [0.4, \, 1.6]$) is still an open problem.

We can also consider eigenvalue perturbation theory for more understanding for the case $\omega \approx 1$.
Since eigenvalues are continuous functions of the elements of the matrix, we know that for $\omega \approx 1$, there will be an eigenvalue near zero.
To be more precise, for $B, C \in \R^{m \times m}$, define the following {\em backward error} for the second argument of a matrix pencil:
\[
\eta(B, C) = \min \, \{ \, \eps \, : \, (B, \, C +\eps \, E) \ {\rm has~an~eigenvalue~1}, \ \|E\|=1 \, \}.
\]
A backward error is useful here to quantify how far a pencil is from a pencil having an eigenvalue equal to 1, such that the $G$-operator has a zero eigenvalue.

\begin{proposition}
We have $\eta(AA\trans, \, L_{\omega}) \le |1-\omega^{-1}| \cdot \max_i \|\ba_i\|^2$.
\end{proposition}
\vspace{-6mm}
\begin{proof}
This is a consequence of
$L_{\omega} = \wh L + \omega^{-1} \, D = (\wh L + D) + (\omega^{-1}-1) \, D$.
This means that $(AA\trans, \, L_{\omega})$ is at most $|1-\omega^{-1}| \cdot \|D\|$ away in the second matrix argument from a pencil with eigenvalue 1.
\end{proof}

We now invoke the general well-known first-order upper bound, as a product of the backward error and the condition number.
This gives that, up to second-order terms in $|1-\omega^{-1}|$, there is an eigenvalue $\lambda$ of $G = G_{\omega}$ such that
\[
|\lambda| \lesssim \kappa(\lambda) \cdot |1-\omega^{-1}| \cdot \max_i \|\ba_i\|^2.
\]
Here, $\kappa(\lambda)$ is the eigenvalue condition number of the eigenvalue $0$ for $\omega=1$, with respect to changes in the second matrix argument.
The quantity $\kappa(\lambda)$ is not information that is computed or approximated by the Kaczmarz method, but we see that especially if $\kappa(\lambda)$ is modest, it is likely that $G_{\omega}$ has a near-zero eigenvalue for $\omega$ close to 1.

To summarize, the Kaczmarz method often has fast initial progress because of (near-)zero eigenvalues, of which there may be several, partly due to good row orderings.

\subsection{Three stages of the iterations, and three factors of influence}
We now give a (new) concise summary of the main effects that influence the convergence of Kaczmarz's method in its various phases.
As an iterative process, we may distinguish three stages of the Kaczmarz iterations:
\mh{
\begin{itemize}
\item The {\em initial} behavior: the first few steps, which is the focus of this paper.
As seen in Section~\ref{sec:zero}, the (nearly) zero eigenvalues are of great importance here.
\item The {\em asymptotic phase}: after (say) $100$ iterations, where the spectral radius is the main factor that determines the behavior.
\item The {\em transient} stage: iterations (say) 10--100, is the phase after the first few iterations and before the typical asymptotic behavior sets in.
Transient behavior is an interesting and challenging topic (cf., e.g., \cite{TEm20} for a study for the matrix exponential), and we will leave it for future research.
In the presence of noise, this stage is often less important in practice, because of the semi-convergence phenomenon, but in this section we study the noise-free case of \eqref{Ax=b}.
\end{itemize}}
We can also list the following three factors that influence the convergence.
None of these factors change the fixed point $\fixp$, but they generally have a (sometimes considerable) effect on the speed of convergence.
It is noteworthy that the impact of these three factors is usually noticeable in both the initial and the asymptotic behavior.
\begin{itemize}
\item The {\em row ordering:} we would like to point out that, interestingly, this ordering has effect on both the early and asymptotic stages.
In view of Proposition~\ref{prop:zero}, the initial convergence for $\omega \approx 1$ is strongly related to the component of $\fixp$ in the direction of the first row $\ba_1$, and possibly other zero eigenvectors.
It will also generally have impact on the spectral radius $\rho(G\vert_{\calv})$ with resulting asymptotic convergence.
Although always $\rho(G\vert_{\calv}) < 1$ for any row ordering (as Kaczmarz is a convergent method), the precise value may depend on the row ordering of $A$. In principle this is relevant for the asymptotic speed of convergence, but since the right-hand side is usually noisy, the practical influence is limited due to the semi-convergence.
\item The {\em relaxation parameter $\omega$:} this effect has been studied extensively in literature.
In the previous subsection we have seen the new observation that $\omega=1$ leads to one or more zero eigenvalues, which is favorable in the first iterations. It also has effect on the spectral radius and the associated asymptotic convergence.
While $\omega=1$ (which is the default value in \textsc{AIR Tools II} \cite{AIRtoolsII}) may not always be optimal for asymptotic convergence, it is never a poor choice in practice.
We will discuss small values of $\omega$ in Section~\ref{sec:small}.
\item The {\em true solution} $\fixp$: for solutions with large components in $\ba_1$ or other (near) zero eigenvectors, the initial convergence will be rapid.
On the other hand, if $\fixp$ is in the direction of the eigenvector corresponding to the eigenvalue associated with $\rho(G\vert_{\calv})$, then both the initial and asymptotic convergence will be very slow.
\end{itemize}
We illustrate these observations with the following examples.

\begin{figure}[htb!]
\centering
\includegraphics[width=0.35\linewidth]{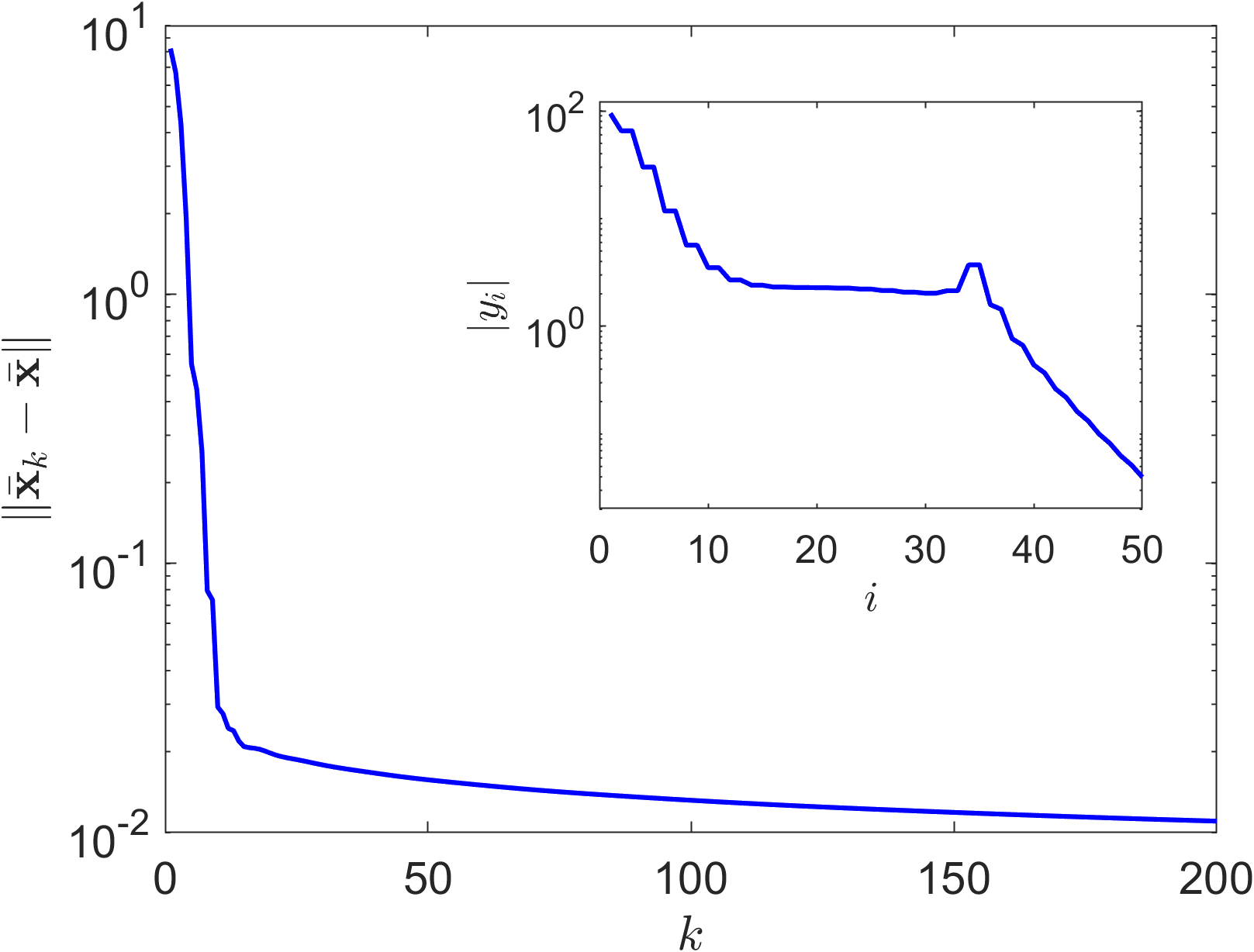}
\caption{Convergence of the Kaczmarz methods for the \textsf{gravity} test problem with ${\mathtt d}=0.06$
and $n=128$. The inset plot shows the absolute values of the elements in $\by = W^{-1} \fixp$, i.e., the
coefficients of the solution in the eigenvector basis.
The fast initial convergence is explained by the fast decay of $|y_i|$.}
\label{fig:Insight}
\end{figure}

\begin{example} \rm
In Figure~\ref{fig:Insight}, the initial convergence is very rapid.
The reason is that the standard solution $\fixp$ provided by \cite{RegTools} has considerable components in the direction of the smallest eigenvectors.
\end{example}

\begin{figure}[htb!]
\centering
\includegraphics[width=0.32\linewidth]{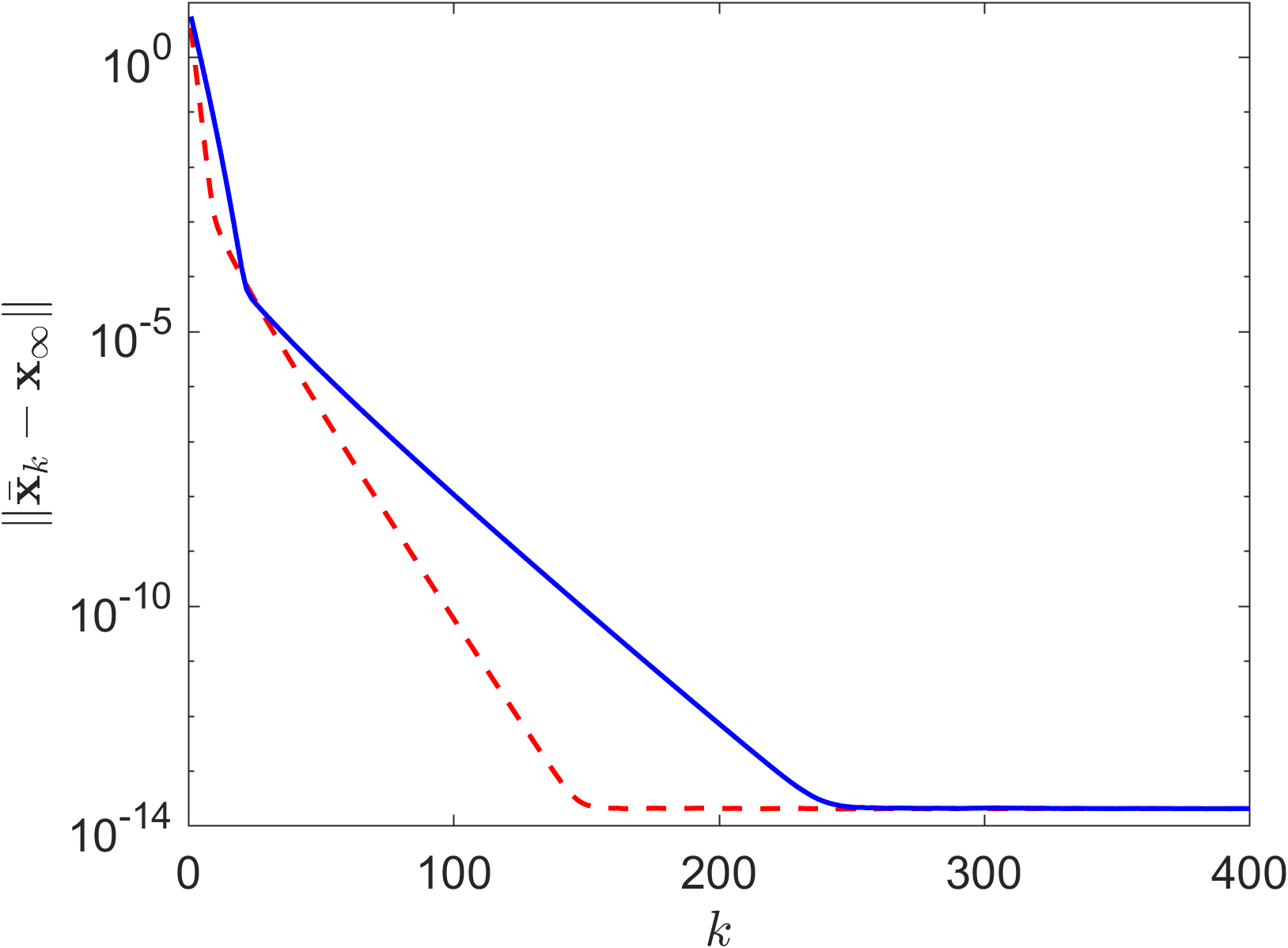}
\includegraphics[width=0.32\linewidth]{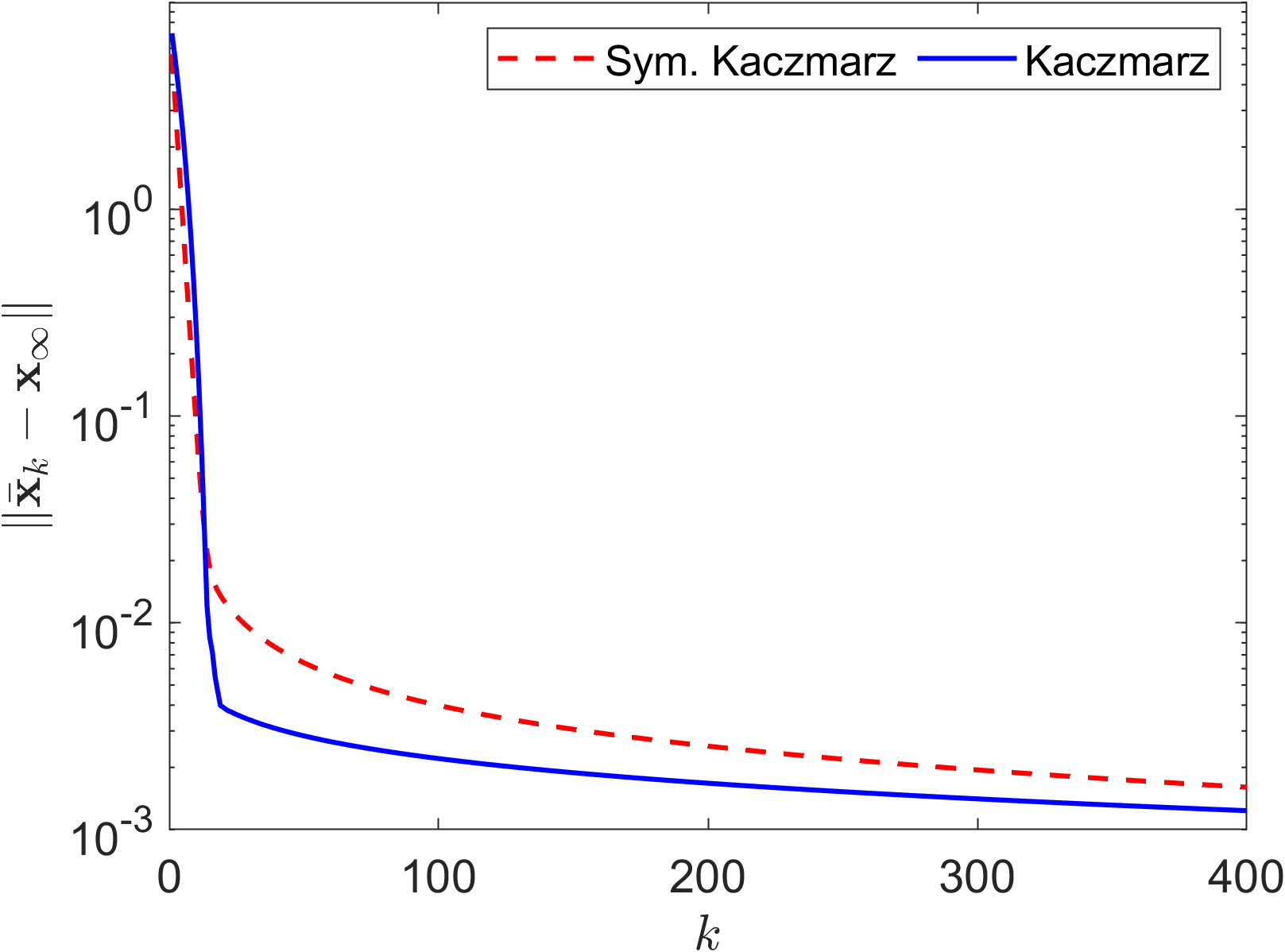}
\includegraphics[width=0.32\linewidth]{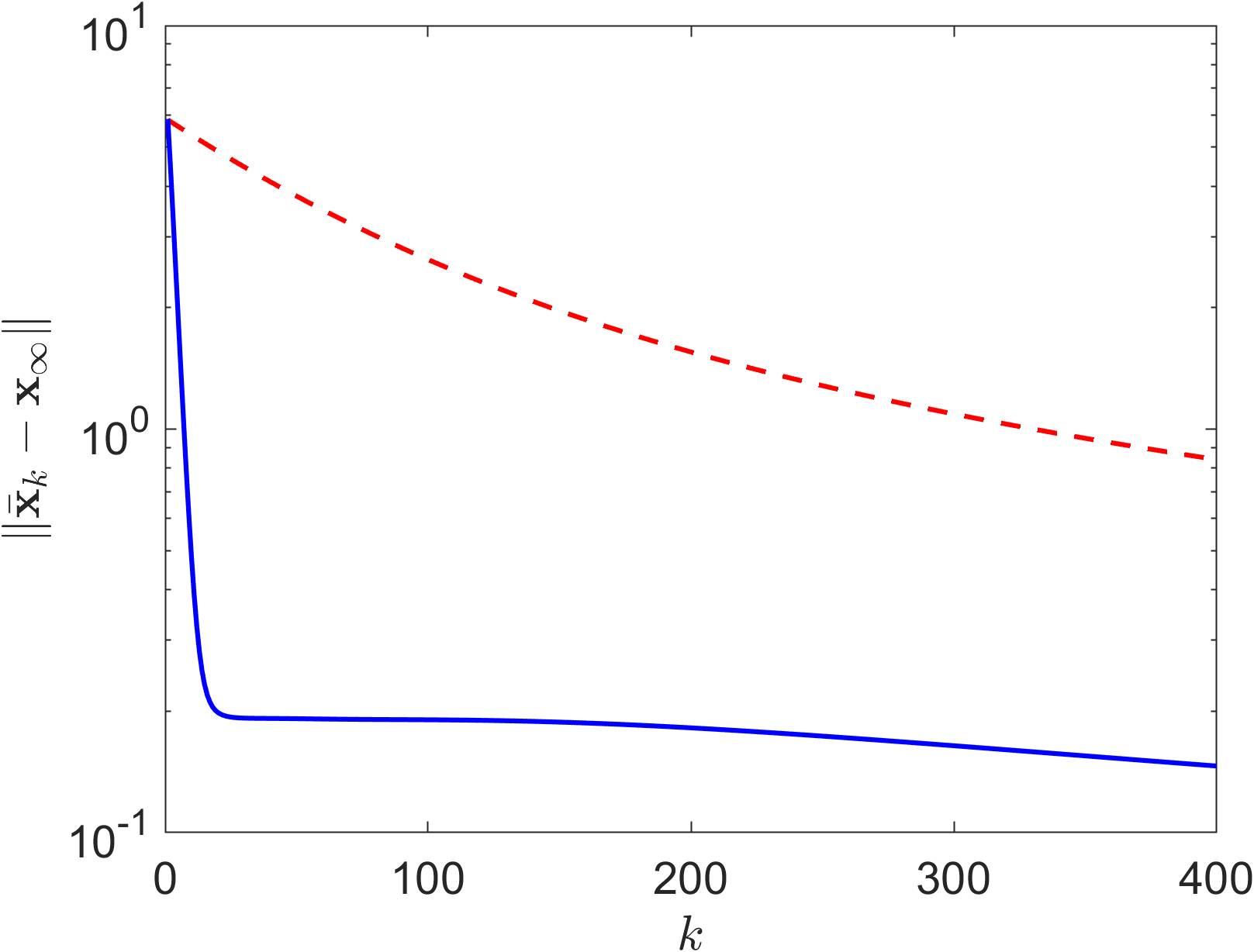}
\caption{Comparison of the Kaczmarz and symmetric Kaczmarz methods
for the {\sf gravity} test problem
with {\tt d} = 0.01 (left), 0.02 (middle), and 0.4 (right).
In all three plots, we see an initial fast convergence followed by a phase with slow asymptotic convergence.}
\label{fig:comparison}
\end{figure}

\begin{example} \rm
Figure~\ref{fig:comparison} shows the convergence of the Kaczmarz and symmetric Kaczmarz methods for noise-free data.
We take {\sf gravity} with $n=128$; we use $\texttt{d}=0.01$, $0.02$, and $0.4$ resulting in condition numbers $10.21$, $415.7$, and $1.95\cdot 10^{19}$, respectively.
We observe the initial fast convergence followed by a (much) slower asymptotic convergence.
We also see that symmetric Kaczmarz can be faster than Kaczmarz, as reflected in the spectral radii:
$\rho(G) \approx 0.92$, $\rho(G_{\rm s}) \approx 0.85$ (for $\mathtt{d}=0.01$),
$\rho(G) \approx 0.9999$, $\rho(G_{\rm s}) \approx 0.9998$ (for $\mathtt{d}=0.02$), and
$\rho(G) \approx \rho(G_{\rm s}) \approx 1$ (for $\mathtt{d}=0.4$).
\end{example}

\begin{figure}[htb!]
\centering
\includegraphics[width=0.35\linewidth]{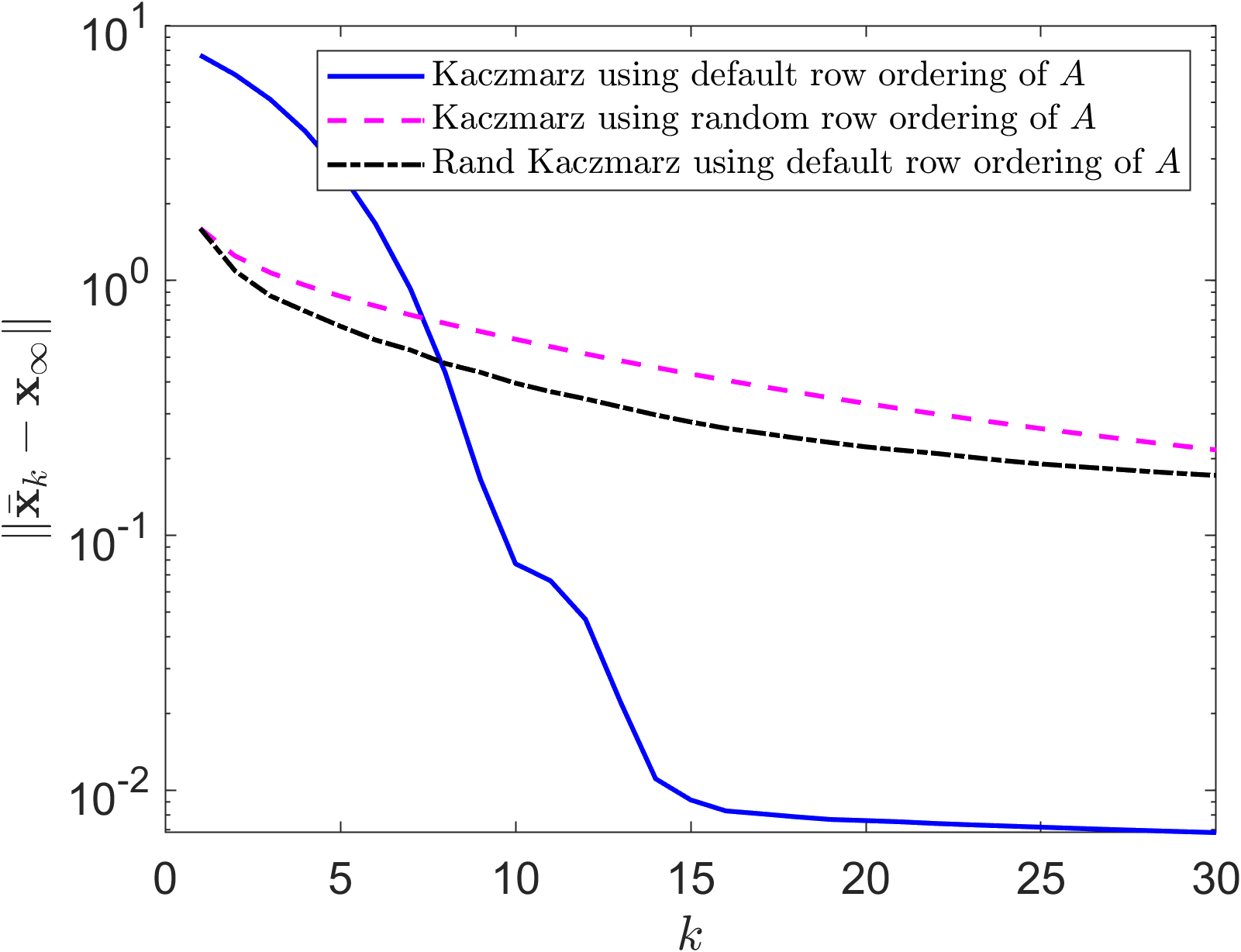}
\caption{The initial iterations of Kaczmarz's method for the {\sf gravity} problem.}
\label{fig:initial}
\end{figure}

\begin{example} \rm
We consider the initial iterations of the Kaczmarz method, using the
\textsf{gravity} test problem with ${\mathtt d} = 0.03$ and noise-free data.
Figure~\ref{fig:initial} shows the iteration errors for Kaczmarz for the default
row order of $A$ and a random row ordering; it is evident that the row ordering can
have a dramatic influence on the initial convergence.
We also show the convergence for the randomized Kaczmarz method, showing that this
method can have slow initial convergence.
The spectral radii (which are not very relevant for the initial behavior) are $\rho(G) \approx 1-10^{-8}$ for both orderings.
\end{example}

\begin{example} \rm
From other unreported experiments, we mention that if the true solution $\fixp$ is equal to the eigenvector corresponding to the smallest eigenvalue in absolute value (so $\fixp = \bw_1 = \ba_1$), then convergence indeed takes place in just one step.
On the other hand, for the other extreme case $\fixp = \bw_n$, corresponding to the spectral radius $\rho(G\vert_{\calv})$, there is virtually no reduction of the iteration error after 200 steps.
To conclude: we can influence the convergence via row orderings and an $\omega$-value, but $\fixp$ is a factor as well, outside of our control.
\end{example}

\subsection{Eigenvalues of the symmetric Kaczmarz method} \label{sec:symK}
As discussed in the introduction, symmetric Kaczmarz has been studied by several authors.
It is less known, and, to the best of our knowledge, the method is not used by practitioners.
In our experience its convergence is often considerably slower than standard Kaczmarz.
One reason to study the symmetric variant is because the analysis for symmetric (and hence normal) operators is much more straightforward than for the nonnormal case, as generally is the case for standard Kaczmarz.
In addition, the hope is that some properties that hold for the symmetric approach, also extend to the standard variant.
Finally, there are some cases where symmetric Kaczmarz is faster than the standard approach (cf.~Figure~\ref{fig:comparison}).

We now study the aspects that influence the convergence of symmetric Kaczmarz.
From the fact that both the standard and symmetric Kaczmarz methods converge (see Section~\ref{sec:stage}), we get the following consequence.

\begin{proposition} \label{prop:Gs}
Let $0 < \omega < 2$.
The spectral radii of the standard and symmetric Kaczmarz satisfy:
\[
\rho(G_{\rm s}\vert_{\calv}) = \|G\vert_{\calv}\|^2, \quad {\rm and} \quad \max \, \{ \, \rho(G\vert_{\calv}), \ \rho(G_{\rm s}\vert_{\calv}), \ \|G\vert_{\calv}\| \, \} < 1.
\]
\end{proposition}
\vspace{-4mm}
\begin{proof}
As both the standard and symmetric Kaczmarz method converge, the two spectral radii are less than one.
Since $\rho(G_{\rm s}\vert_{\calv}) = \|G\vert_{\calv}\|^2$, this should also hold for $\|G\vert_{\calv}\|$.
\end{proof}

\begin{example} \rm
We give an example of the smallest possible dimension $m=n=2$, which gives an intuition of the behavior of $G\vert_{\calv}$.
Consider the $2 \times 2$ matrix
\[
G = \left[
\begin{array}{cc} 0.99 & \alpha \\ 0 & 0.98 \end{array}\right]
\]
for small $\alpha \ge 0$; which is intended to display the typical behavior of $G$ in miniature format.
Obviously, this matrix has $\rho(G) = 0.99$.
It can be checked numerically that $\|G\| < 1$ if and only if $\alpha < 0.0281$.
Since $\rho(G\trans G) = \|G\|^2$, this is also exactly the condition such that $\rho(G_{\rm s}) < 1$.
For this matrix, $\alpha$ is one possible measure of nonnormality (cf.~\cite{EPa87}).
Informally, this example illustrates that in Kaczmarz method, since the properties in Proposition~\ref{prop:Gs} hold, $G$ can only be modestly nonnormal.
\end{example}

The quantity $\rho(G_{\rm s}\vert_{\calv})$ from Proposition~\ref{prop:Gs} is related to the asymptotic convergence.
As in Section~\ref{sec:zero}, similar to the $G$ in the standard method, the symmetric Kaczmarz operator $G_{\rm s} = G\trans G$ also has $\ba_1$ as a zero eigenvector for $\omega = 1$ (as this is also a singular vector corresponding to a zero singular vector of $G$).
Therefore, just as for the standard approach, we also expect fast initial convergence in the symmetric method when $\fixp$ has a considerable component in the (near-)zero eigenvector(s).

\subsection[]{Upper bounds for $\rho(G\vert_{\calv})$}
As discussed in the previous subsection, $\|G\vert_{\calv}\| < 1$ is an upper bound for $\rho(G\vert_{\calv})$.
We will now present two new alternative upper bound for this spectral radius, for which we first need some preparations.
Consider the symmetric part of $L^{-1}$, that is, $\frac12 \, (L^{-1}+L^{-\top})$.
All eigenvalues of this symmetric part are real; denote the smallest eigenvalue by $\nu(L^{-1})$.
It is the smallest real part of the field of values $\calf(L^{-1})$ (see, e.g., \cite{TEm20}), where
\[
\calf(L^{-1}) = \{ \, \bz^* L^{-1} \, \bz  \, : \, \bz \in \C^m, \ \|\bz\| = 1 \, \}\, ;
\]
that is, $\nu(L^{-1}) = \min \, \{ \, {\rm Re}(\zeta) \, : \, \zeta \in \calf(L^{-1}) \, \}$.
Since
\[
L+L\trans = AA\trans + (2 \, \omega^{-1} - 1) \, D,
\]
we conclude that $L+L\trans$ is symmetric positive definite (SPD) for $0 < \omega < 2$.
Then the similarity transform $L^{-1} \, (L+L\trans) \, L^{-\top} = L^{-1} + L^{-\top}$ is also SPD.
This means that $\nu(L^{-1}) > 0$.
(Since the symmetric part is positive definite, $L^{-1}$ is sometimes also called positive definite although it is not symmetric.)

For most problems, the eigenvalue corresponding to the spectral radius $\rho(G\vert_{\calv})$ is real.
This is not always the case; as an example, we mention that {\sf gravity} of dimension $n=128$ with parameters ${\mathtt d} = 0.01$ and $\omega = 1.4$ gives a complex conjugate pair with maximum absolute value.
However, in most practical cases, $\rho(G\vert_{\calv})$ corresponds to a real and positive eigenvalue close to 1 (such as for ${\mathtt d} > 0.012$ in this example).
Hence, the assumption in the following result is common.

\begin{proposition}  \label{prop:twobounds}
Let $0 < \omega < 2$.
Suppose $\rho(G\vert_{\calv})$ is associated with a simple and real eigenvalue $\lambda > 0$.
Then
\[
\rho(G\vert_{\calv})
\ \le \ 1 - \frac{\sigma_{\min}^2(A\vert_{\calv})}{\|L\|}
\ \le \ 1 - \nu(L^{-1}) \cdot \sigma_{\min}^2(A\vert_{\calv}) \, .
\]
\end{proposition}
\begin{proof}
Let $\bx \in \calv$ be the eigenvector associated  with $\lambda$, such that $\lambda = \bx^* G\, \bx$.
Because of the assumptions, $\rho(G\vert_{\calv}) = \lambda = 1 - (\bx^* A\trans) \, L^{-1} \, (A \bx)$.
Then the first bound follows from
\[
(\bx^* A\trans) \, L^{-1} \, (A \bx) \ge
\sigma_{\min}(L^{-1}) \cdot \sigma_{\min}^2(A\vert_{\calv}) \cdot \|\bx\|^2
= \|L\|^{-1} \cdot \sigma_{\min}^2(A\vert_{\calv}),
\]
We also have
\[
(\bx^* A\trans) \, L^{-1} \, (A \bx) \ge \nu(L^{-1}) \cdot \|A\bx\|^2
\ge \nu(L^{-1}) \cdot \sigma_{\min}^2(A\vert_{\calv}).
\]
which yields the second bound.
Finally, let $\by$ be the largest singular vector of $L$.
By taking $\bz = L\by \, / \, \|L\by\|$ in $\nu(L^{-1}) \le \bz^* \, L^{-1} \bz$, it may be checked that $\nu(L^{-1}) \le \|L\|^{-1}$, so that the first bound is as least as tight.
\end{proof}

For ill-posed problems, $\sigma_{\min}^2(A\vert_{\calv})$ may be (very) close to zero, so this upper bound may be (very) close to 1.
However, this is realistic, since $\rho(G\vert_{\calv})$ is often (very) close to 1.

\begin{example} \rm
For {\sf gravity} with $n=128$, ${\mathtt d} = 0.02$ and $\omega = 1$, the matrix $A$ has no nullspace, so no restriction is necessary.
We have $\rho(G) \approx 1-1 \cdot 10^{-4}$.
The two upper bounds for $\rho(G)$ from Proposition~\ref{prop:twobounds} are both approximately $1-1 \cdot 10^{-5}$.
The norm upper bound is the best of the three bounds: $\rho(G) \le \|G\| \approx 1-9.9 \cdot 10^{-5}$.
\end{example}

\subsection[ ]{Spectral properties for small values of $\omega$}  \label{sec:small}
From our extensive practical experience, the authors generally recommend an $\omega$-value between 0.4 and 1.8; as mentioned, $\omega=1$ is a good default choice for many problems.
To add to the overall picture of $G$, in this section we investigate some spectral properties of $G$ for small $\omega$, a value which also may have its merits.
An early study \cite{CEG83} on this situation has already pointed out that the limit solved a weighted least squares problem.
For $\omega \to 0$, it is easy to check that
\[
A\trans L_{\omega} \, A = A\trans (\wh L + \omega^{-1} D)^{-1} \, A = \omega \ A\trans D^{-1} \, A - \omega^2 \, A\trans D^{-1} \, \wh L \, D^{-1} \, A  +\calo(\omega^3).
\]
Disregarding second- and higher-order terms, the operator $I-\omega \, A\trans D^{-1} A$ is the first-order part of $G$; it is symmetric positive definite.
For small $\omega$ we have the first-order upper bound for its largest eigenvalue, which determines the spectral radius:
\[
\rho((I-\omega \, A\trans D^{-1} A)\vert_{\calv}) \lesssim 1 - \omega \cdot \sigma_{\min}^2(A\vert_{\calv}) \cdot (\max_i \|\ba_i\|^2)^{-1}.
\]
Just as in the previous subsection, this upper bound may be very close to 1, but this may also hold for the spectral radius itself.

When $\|\omega \, A\trans L^{-1} A\| < 1$, all eigenvalues of the first-order part of $G$  have positive real part.
Since $\|A\trans L^{-1} A\vert_{\calv}\| \ge \sigma_{\min}(A\trans L^{-1} A\vert_{\calv})$,
this means that we expect that this is the case for all $\omega \lesssim \|(A\trans L^{-1} A\vert_{\calv})^{-1}\|$.

Moreover, we can use the fact that eigenvalues are continuous functions of the matrix elements.
Under the very mild assumption that the eigenvalues of $A\trans D^{-1} A$ are simple (which is generically the case), there exists an $\omega_0 > 0$, such that for all $0 < \omega < \omega_0$, all eigenvalues of $G$ are real.
This is illustrated in Figure~\ref{fig:zero}.
where we see that $\omega_0 \approx 0.08$ for the {\sf gravity} problem while $\omega_0 \approx 0.004$ for the tomography problem.

\begin{figure}[htb!]
\centering
\includegraphics[width=0.3\linewidth]{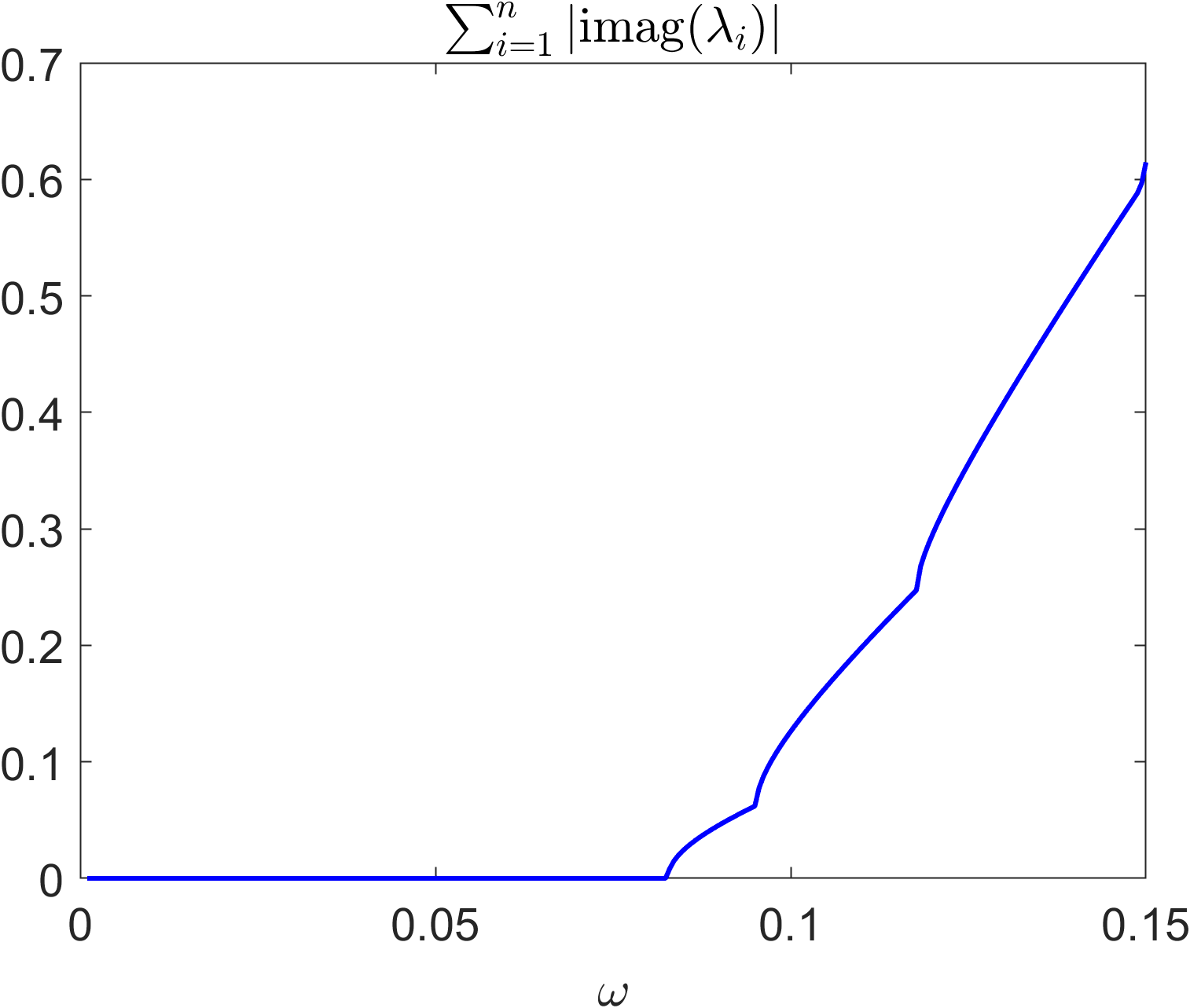} \quad
\includegraphics[width=0.3\linewidth]{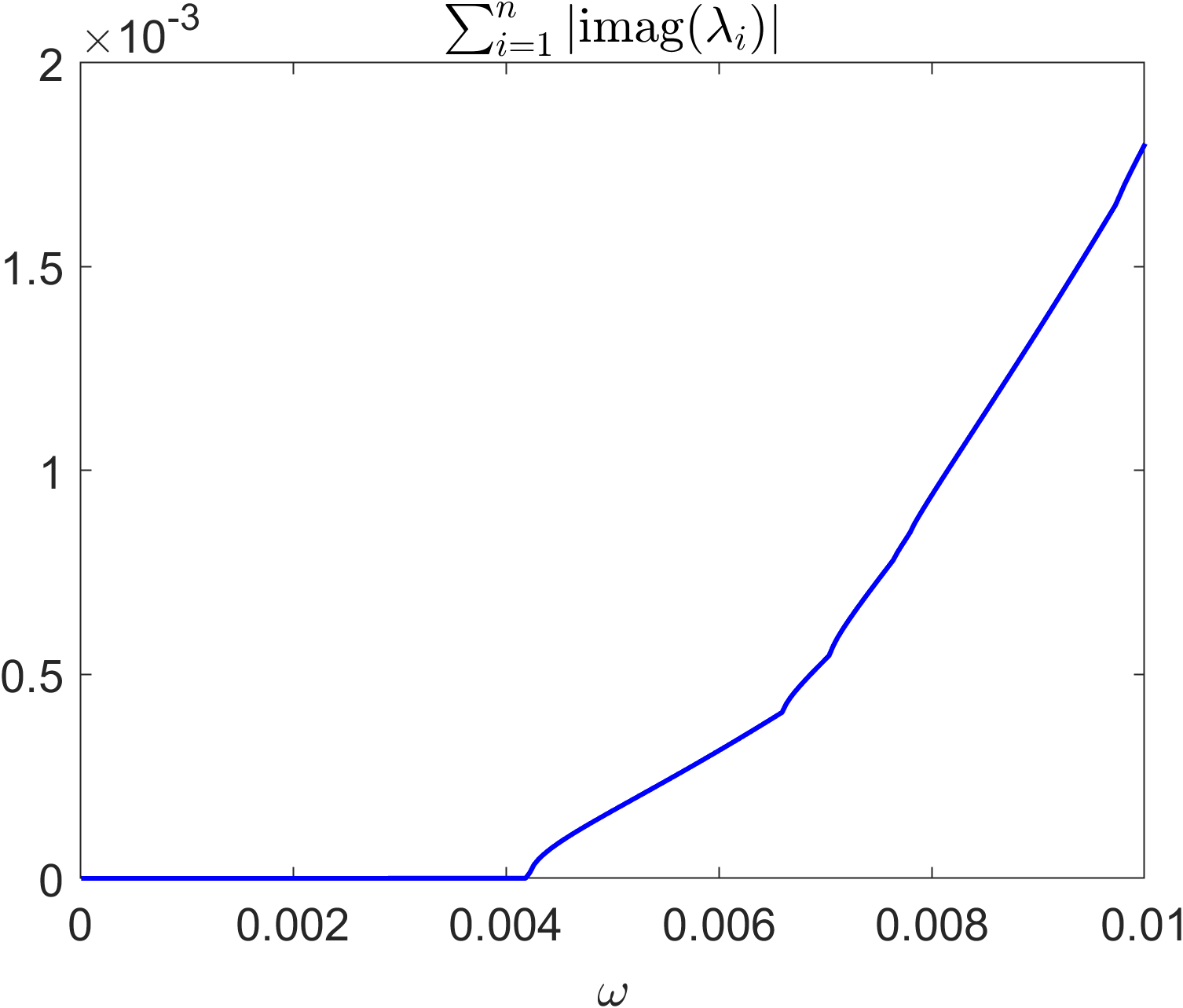}
\caption{Two examples illustrating that all eigenvalues are real for all $\omega$ smaller than a certain $\omega_0$.
Left:\ \textsf{gravity} problem with ${\mathtt d} = 0.06$ and $n=128$.
Right:\ tomography problem with $m=n=1024$.
All eigenvalues are real for small $\omega$.}
\label{fig:zero}
\end{figure}

To summarize the aspects discussed above, for small $\omega$ the operator $G$ behaves as a symmetric map (cf.~also \cite{CEG83}), with only real eigenvalues that approach 1 in the limit.
The convergence of symmetric Kaczmarz is often slow, and spends two sweeps per iteration.
Standard Kaczmarz with a small value of $\omega$ (e.g., $0.1$--$0.2$) results in an almost symmetric operator without the double sweeps of symmetric Kaczmarz, with hopefully faster convergence, while it also nearly produces a least squares solution (cf.~\cite{CEG83}).
Having only eigenvalues of $G$ with positive real part may also result in fewer oscillations of the iteration error.
This suggests that if a careful approach is preferred by practitioners, {\em standard Kaczmarz with a modest $\omega$ may combine some of the advantages of the standard and symmetric approaches}.
However, this small $\omega$ may generally not be the parameter value that leads to the fastest convergence.

\medskip\noindent
In this Section~\ref{sec:eigen}, we have reviewed several spectral aspects to obtain a better understanding of the convergence of Kaczmarz's method, especially for the initial iterations.
For the asymptotic situation, we did not include the role of the noise in this section; this is the topic of Section~\ref{sec:stat}.

\section{Statistics of the noise propagation} \label{sec:stat}

To obtain more insight into the convergence and semi-convergence of Kaczmarz's method
we now study how the noise in the right-hand side propagates to the iterates.
Classical semi-convergence analysis (see \cite{Elfving25} for an overview)
provides a (sometimes pessimistic) upper bound for the norm of the noise
component in $\bx_k$,
while a lower bound for this norm is lacking.
Our statistical analysis, following the approach in \cite{HansenLAA},
uses a different viewpoint to provide insight that semi-convergence is very likely to happen.

In this paper, for convenience, we consider white Gaussian noise in the right-hand side which we write as
  \begin{equation} \label{eq:setup}
    \bb = \bar \bb + \be \ , \qquad \bar \bb = A\, \bar \bx \ , \qquad
    \be \sim \caln(0,\, \sigma^2 I) \ ,
  \end{equation}
where $\bar \bx$ is the exact solution (also called the ground truth),
$\bar \bb$ is the noise-free data, $\be$ is the noise, and $\sigma$ is its element-wise
standard deviation.
Again, we assume that $\bar \bx$ has no component in the nullspace $\calv^{\perp}$ of $A$
since we are not able to reconstruct such a component (this is a general difficulty
for inverse problems).
This means that if there is no noise then the fixed point $\fixp$ is identical to
the exact solution $\bar{\bx}$.

With $A_k^{\#}$ given by \eqref{Ak}, we can write the $k$th iterate as
  \[
    \bx_k = A_k^{\#} \, \bb = \bar \bx_k + A_k^{\#} \, \be \qquad \hbox{where} \qquad
    \bar \bx_k = A_k^{\#} \,  \bar \bb \ ;
  \]
we refer to $\bar \bx_k$ as the noise-free iterates.
As usual (see, e.g., \cite{ENH10}), we can then split the reconstruction error into two components:
  \begin{equation} \label{eq:three}
    \bx_k - \bar \bx = \underbrace{\bx_k - \bar \bx_k}_{\hbox{noise error}}
    + \underbrace{\bar \bx_k - \bar \bx}_{\hbox{it.~error}} .
  \end{equation}
We refer to $\bar \bx_k - \bar \bx$ as the \emph{iteration error}; in Section~\ref{sec:eigen}
we have derived results for its initial behavior.
The other component $\bx_k - \bar \bx_k = A_k^{\#} \, \be$, which we refer to as the \textit{noise error},
describes the propagated noise from the data.
A general observation \cite{EHN14} is that the noise error increases with the number of iterations and
dominates in the later iterations while it is small in the initial iterations.
Consequently, the reconstruction error in the initial iterations is practically independent of the noise.
Our aim here is to substantiate this observation by studying the statistics of $\| A_k^{\#} \be \|$.

A convenient way to do this is in the eigenvector basis.
\mh{In terms of the matrices $\Afix$ and $A_k^{\#}$ defined in \eqref{eq:fixpoint} and \eqref{Ak}, respectively,
we can write the noise error as
  \[
    \bx_k - \bar \bx_k = A_k^{\#} \be = (I-G^k)\, \Afix \be =
    (I-G^k)\, (A\trans L^{-1} \, A\vert_{\calv})^{-1} \, A\trans L^{-1} \, \be \ ,
  \]
This is identical to \cite[Eqs.~(16) and (18)]{Elfving25}; we will now add an analysis using the spectral decompositions in \eqref{eigG} and \eqref{eq:ED}. We can write
  \[
    A_k^{\#} \be = W\, (I-\Lambda^k)\, W^{-1} \Afix \be \ .
  \]
Moreover, if we introduce
  \[
    \bxi^k = W^{-1} A_k^{\#} \, \be \ , \qquad \bxi = W^{-1} A^{\#} \, \be
  \]
then we have
  \[
    \bxi^k = (I-G^k)\,\bxi \ .
  \]
Then it follows that}
  \begin{equation} \label{eq:normxik}
    \bigl\| \bxi^k \bigr\|^2 = \bigl\| (I-G^k)\, \bxi \bigr\|^2 =
    \sum_{i=1}^n \ \bigl| 1 - \lambda_i^k \bigr|^2 \ |\xi_i|^2 \ ,
  \end{equation}
in which $\xi_i$ are the elements of the vector $\bxi$.
Notice that $k$, the number of iterations, only enters via the factor $I-G^k$.
Neither $\| A_k^{\#} \be \|$ nor $\| \bxi^k \|$ increases monotonically with $k$
in general (as may be easily verified by numerical experiments).
For the latter norm, this follows from the fact that for complex $\lambda_i$ with $|\lambda_i|<1$, we still may have $|1-\lambda_i^k|>1$.

\begin{figure}[htb!]
\centering
\includegraphics[width=0.4\linewidth]{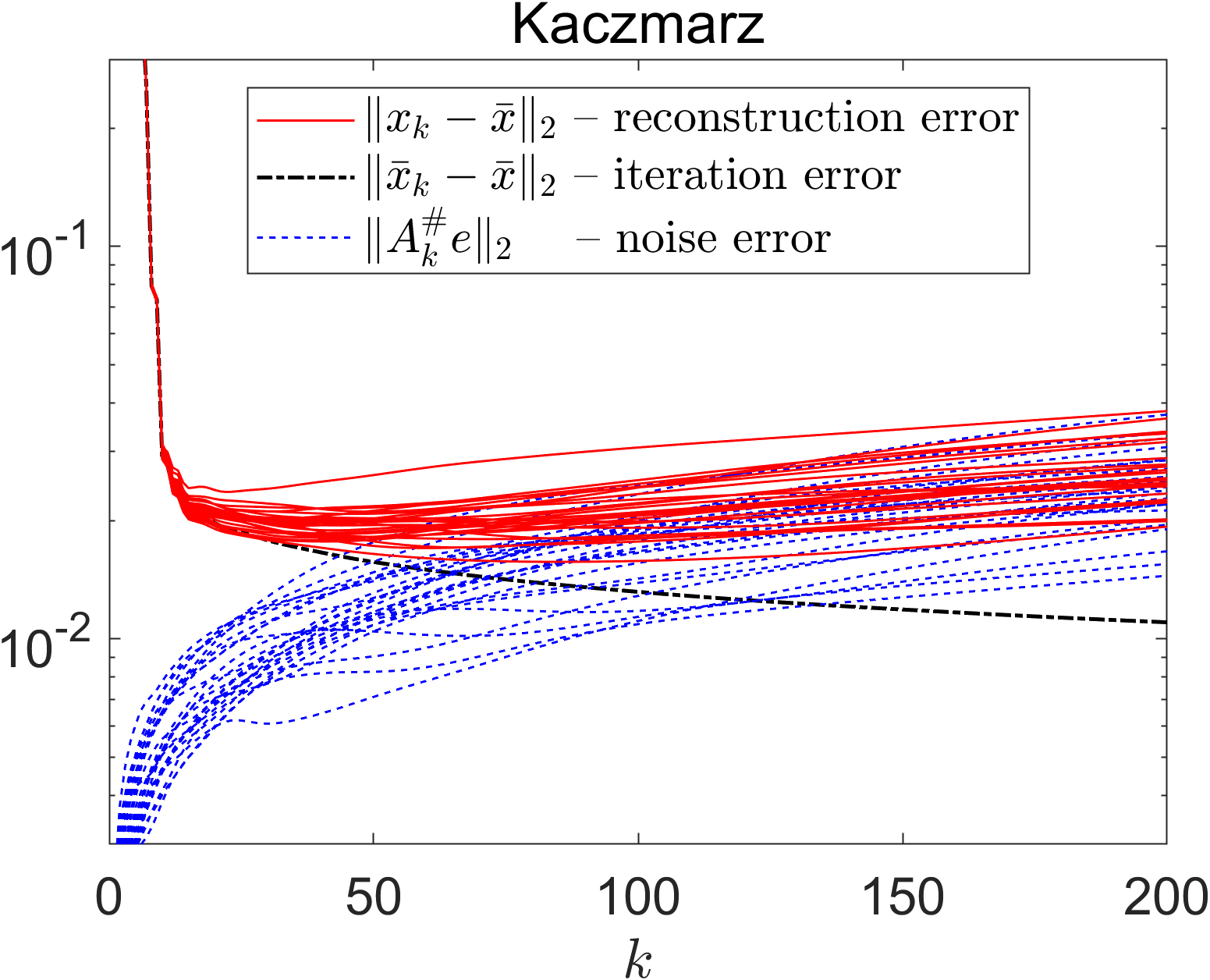}
\includegraphics[width=0.4\linewidth]{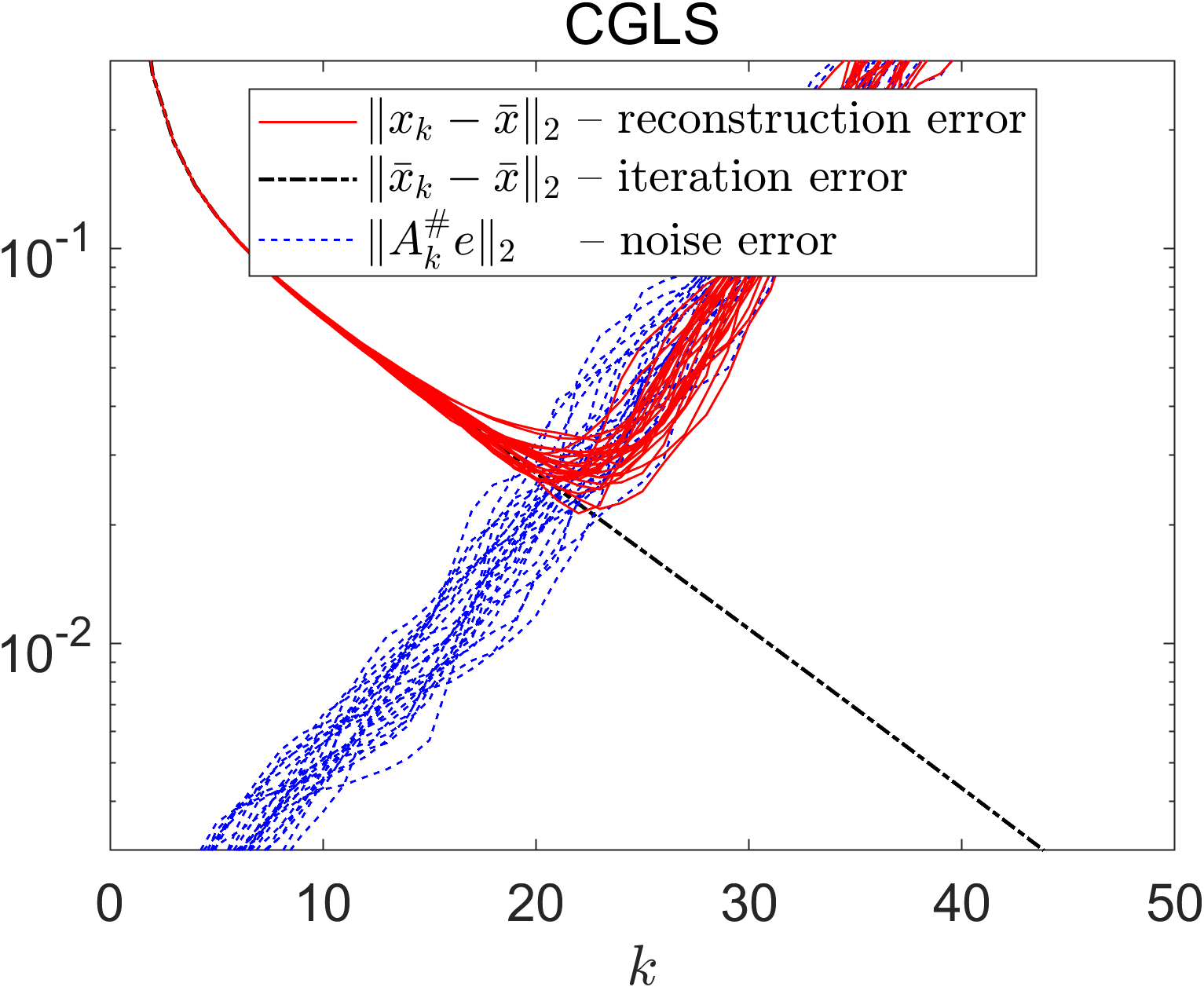}
\caption{Illustration of the three errors in \eqref{eq:three} for the {\sf gravity}
test problem, and with 25 realization of the noise.
The results for CGLS (whose behavior is analyzed in \cite{HansenLAA}) are included to demonstrate
that the underlying behavior of semi-convergence is the same for both methods.}
\label{fig:gravity}
\end{figure}

\begin{example} \rm
To illustrate the mechanism underlying semi-convergence, Figure~\ref{fig:gravity} shows plots of the norms
of the three errors defined in \eqref{eq:three}: the reconstruction error $\bx_k-\bar \bx$,
the iteration error $\bar \bx_k-\bx_k$, and the iteration error $\bx_k-\bar \bx_k$.
We use the test problem {\sf gravity} with ${\mathtt d} = 0.06$ and $n=128$,
and we show results for both CGLS (whose behavior is analyzed in \cite{HansenLAA}) and Kaczmarz,
for 25 realizations of the noise.
A single curve for the iteration error is independent of the noise.
All 25 reconstruction errors behave qualitatively in the same way:\ at first this error
is dominated by the decreasing iteration error (which initially decreases very fast for Kaczmarz);
later it is dominated by the increasing noise error.
Note that in a few cases, the norm of the noise error does not increase monotonically.
\end{example}

\begin{figure}[htb!]
\centering
\includegraphics[width=0.5\linewidth]{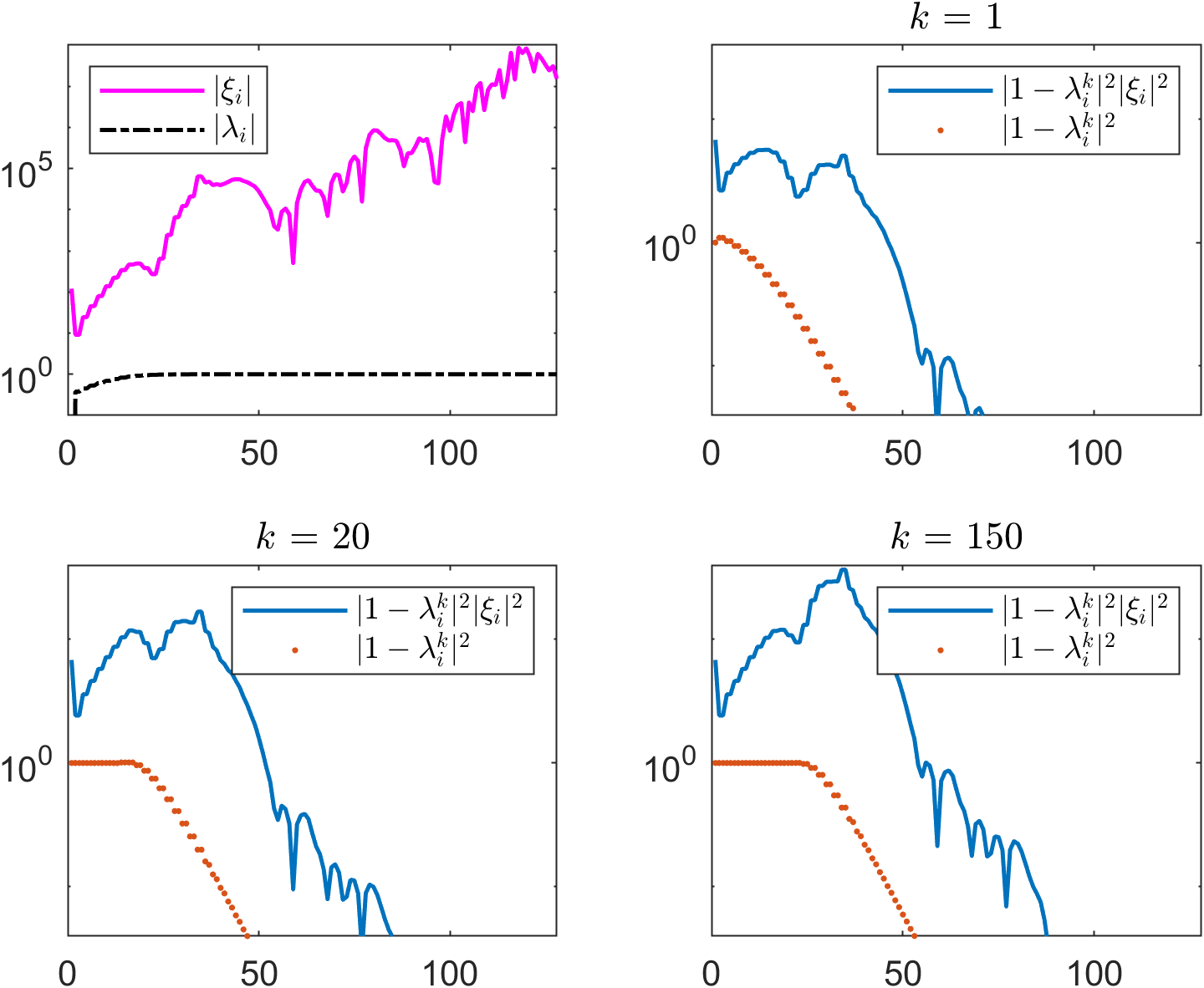}
\caption{The ingredients of $\| \bxi^k \|^2$ in \eqref{eq:normxik} for the {\sf gravity}
test problem of size $n=128$, illustrating an increasing amount of noise with the number of iterations.}
\label{fig:KaczmarzGravityAnalysis}
\end{figure}

\begin{example} \rm
In this example we illustrate the ingredients of $\| \bxi^k \|^2$ in \eqref{eq:normxik}
using the {\sf gravity} test problem
of size $n=128$ and with noise standard deviation $\sigma = 5\cdot 10^{-3}$.
Figure~\ref{fig:KaczmarzGravityAnalysis} shows the absolute values of the eigenvalues of $G$
(note that $\lambda_1=0$)
and the absolute values of components $\xi_i$ of the vector $\bxi = W^{-1} A^{\#} \be$.
The latter values increase, in average, as $i$ increases and the corresponding $|\lambda_i|$
decreases.
The components of large magnitude correspond to high-frequency components in $\bxi$ due to
the noise amplification that is common in inverse problems.
Figure~\ref{fig:KaczmarzGravityAnalysis} also shows, for three values of the number of iterations $k$,
the factors $|1-\lambda_i^k|^2$ and the terms $|1 - \lambda_i^k|^2 \ |\xi_i|^2$
of the sum in \eqref{eq:normxik}.
We see that as the number of iterations increases, more factors  $|1-\lambda_i^k|^2$
get close to 1 and more terms $|1 - \lambda_i^k|^2 \ |\xi_i|^2$ become large.
This illustrates why the norm $\| \bxi^k \|^2$ tends to increase with~$k$.
\end{example}

Now we turn to the statistical aspects.
The following result gives insight about the expected value of these norms.

\begin{proposition} \label{prop:expected}
With the assumptions in \eqref{eq:setup} the expected value related to the noise error satisfies
  \begin{equation} \label{eq:E1}
    \ex(\| A_k^{\#} \,  \be \|^2) = \sigma^2 \ \| A_k^{\#} \|_{\rm F}^2
    = \sigma^2 \ \| W\, \Phi_k W^{-1} \Afix \|_{\rm F}^2
  \end{equation}
and
  \begin{equation} \label{eq:E2}
    \ex(\|\bxi^k\|^2) = \sum_{i=1}^n \ |1 - \lambda_i^k|^2 \ \ex(|\xi_i|^2) \ ,
  \end{equation}
where
  \[
    \Phi_k = I - \Lambda^k = \mathrm{diag} \, (1 - \lambda_i^k) \ .
  \]
\end{proposition}

\begin{proof}
If $\be$ has mean $\bmu$ and covariance matrix $\Sigma$, and if $A$ is SPD, then $\ex(\be\trans Z\, \be) = \bmu\trans \! Z\, \bmu + \tr(Z\,\Sigma)$ \cite[Ch.~4]{qforms}.
Therefore,
  \[
    \ex\bigl(\| A_k^{\#} \,  \be \|^2\bigr) = \tr\bigl((A_k^{\#})\trans A_k^{\#} \, \mathrm{cov}(\be)\bigr)
    = \sigma^2 \ \tr\bigl((A_k^{\#})^* A_k^{\#}\bigr) = \sigma^2 \ \| A_k^{\#} \|_{\rm F}^2 \, ,
  \]
where we use the relation between trace and the Frobenius norm.  Similarly,
  \[
    \ex\bigl(\|\bxi^k\|^2\bigr) = \tr\bigl( \Phi_k^* \, \Phi_k \ \mathrm{cov}(\bxi) \bigr)
  \]
and we use that the diagonal elements of $\Phi_k^* \, \Phi_k \, \mathrm{cov}(\bxi)$ are
$(1-\lambda_i^k)^*(1-\lambda_i^k) \ \ex(|\xi_i|^2)$.
\end{proof}

\begin{figure}
\centering
\includegraphics[width=0.6\linewidth]{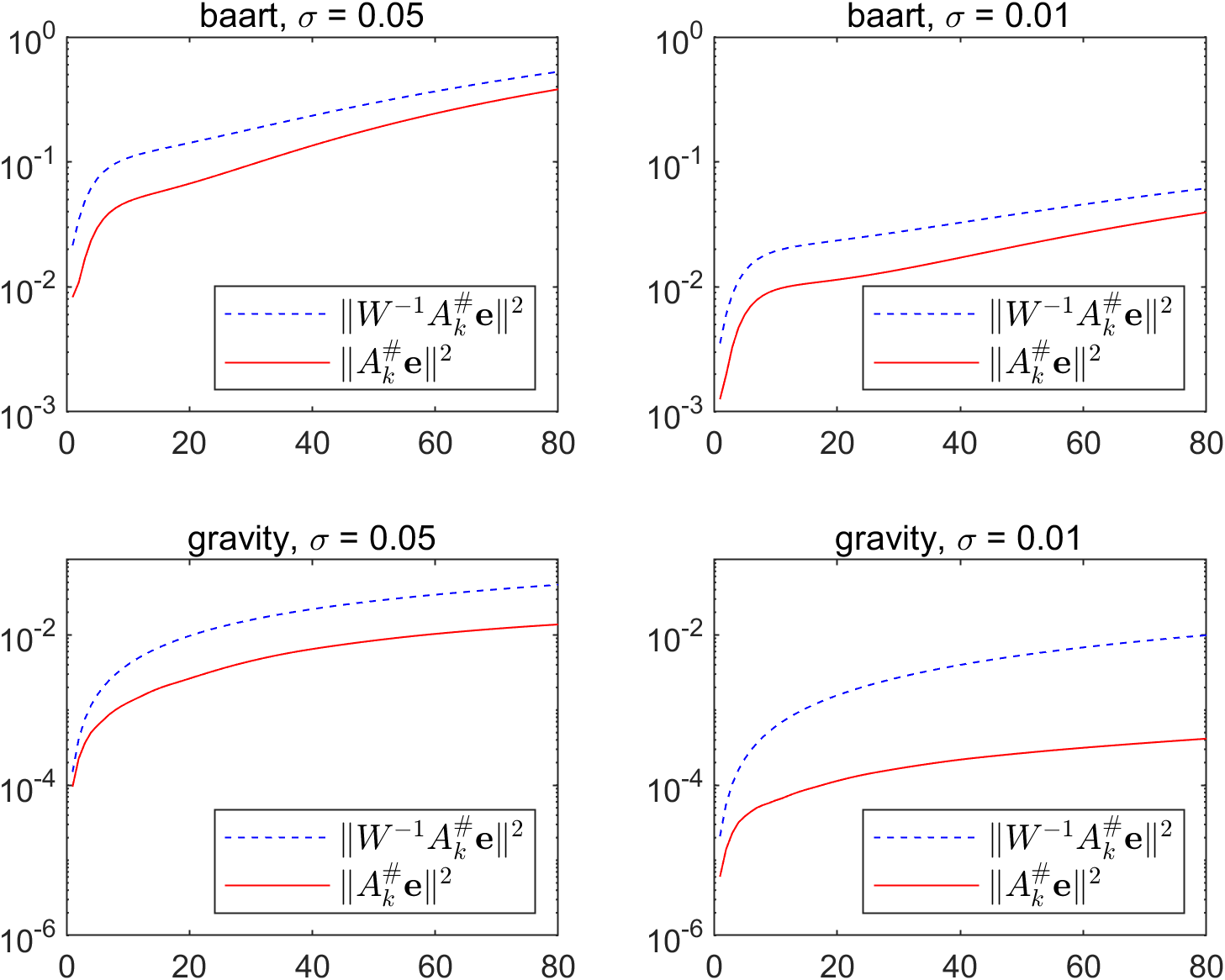}
\caption{Expected values of noise errors for two problems and two noise levels,
tracking each other quite well.}
\label{fig:August9ErrorPropagation}
\end{figure}

\begin{example} \rm
Proposition~\ref{prop:expected} above gives expressions for the expected value of the squared norm
of both the noise error $A_k^{\#}\be$ and the vector $\bxi^k$ of its coefficients
in the eigenvector basis.
Figure~\ref{fig:August9ErrorPropagation} shows that the two norms often tend to ``track''
each other and thus gives the same overall information about the growth of the noise error.
We show results for two test problems, {\sf gravity} and {\sf baart} and two different noise levels.
\end{example}

For the symmetric Kaczmarz method, the eigenvalues of $G\trans G$ are real and less than one in absolute value (cf.~Section~\ref{sec:symK}).
Hence, both expected values in Proposition~\ref{prop:expected} increase monotonically with $k$, and
their behavior follows the theory developed in \cite[Sec.~4]{HansenLAA}.

The situation is different for the ordinary Kaczmarz method.
In many numerical experiments we observe that
$\ex(\| A_k^{\#} \,  \be \|^2)$ as well as $\ex(\|\bxi^k\|^2)$ increase monotonically with $k$,
but this is not always the case.
As seen, the matrix $G$ typically has many complex eigenvalues~$\lambda_i$,
and therefore we cannot guarantee that the corresponding factors $|1-\lambda_i^k|$ increase
monotonically with~$k$ (plots of $|1-\lambda_i^k|$ versus $k$ have one or more
bumps, and the number of bumps increases as $|\lambda_i|$ increases towards 1).
Thus, in turn, we cannot ensure that the expected values behave monotonically.

However, recall that the expected values in Proposition~\ref{prop:expected}
involve all $n$ eigenvalues of $G$, and note that the bumps of $|1-\lambda_i^k|$,
for a fixed $k$, usually occur for different indices $i$.
The contributions from all $\lambda_i$ to the expected values in \eqref{eq:E1} and \eqref{eq:E2} therefore
often tend to balance each other, which explains the (almost) monotonic behavior that we often observe
in practice.

\begin{figure}[htb!]
\centering
\includegraphics[width=0.3\linewidth]{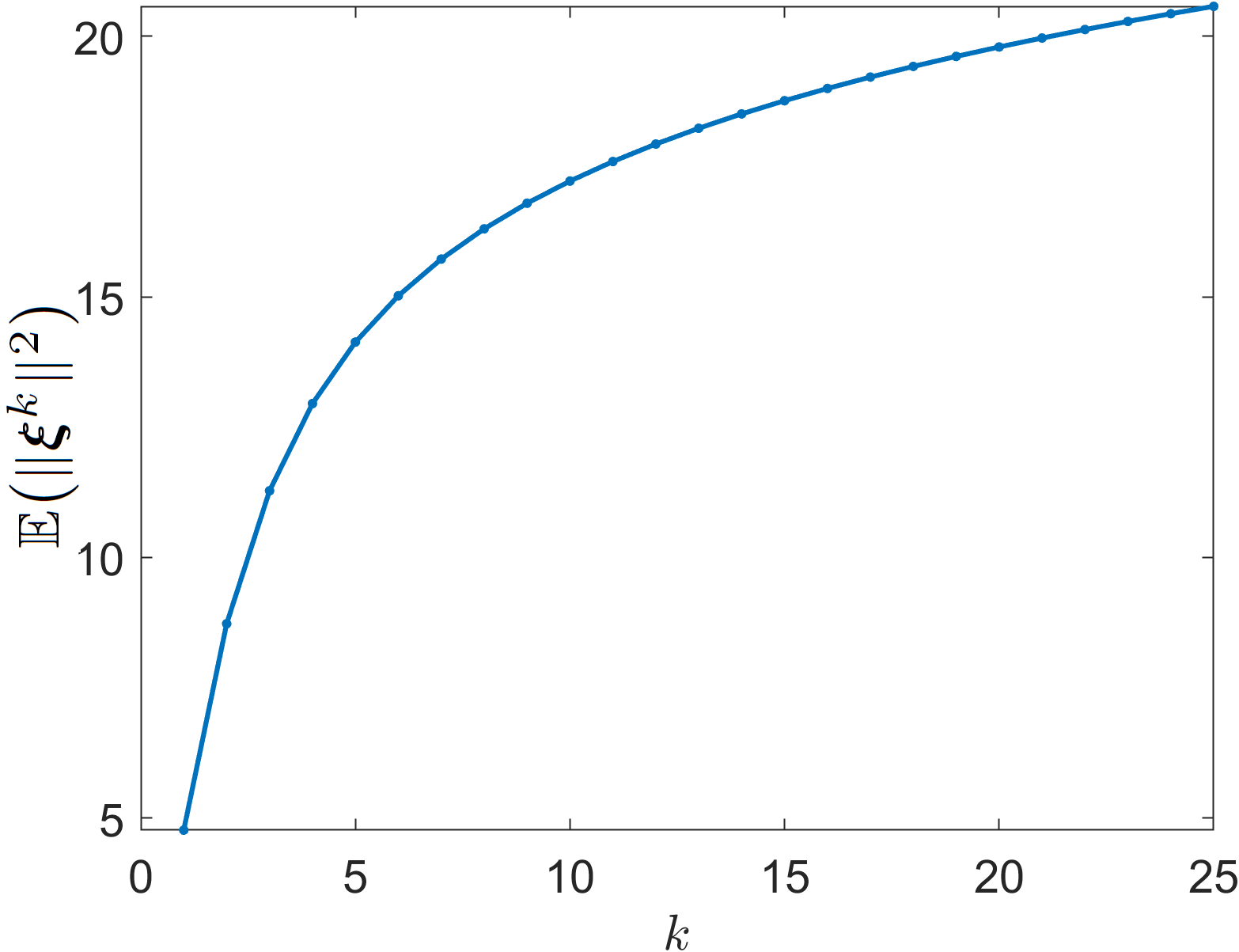}
\includegraphics[width=0.32\linewidth]{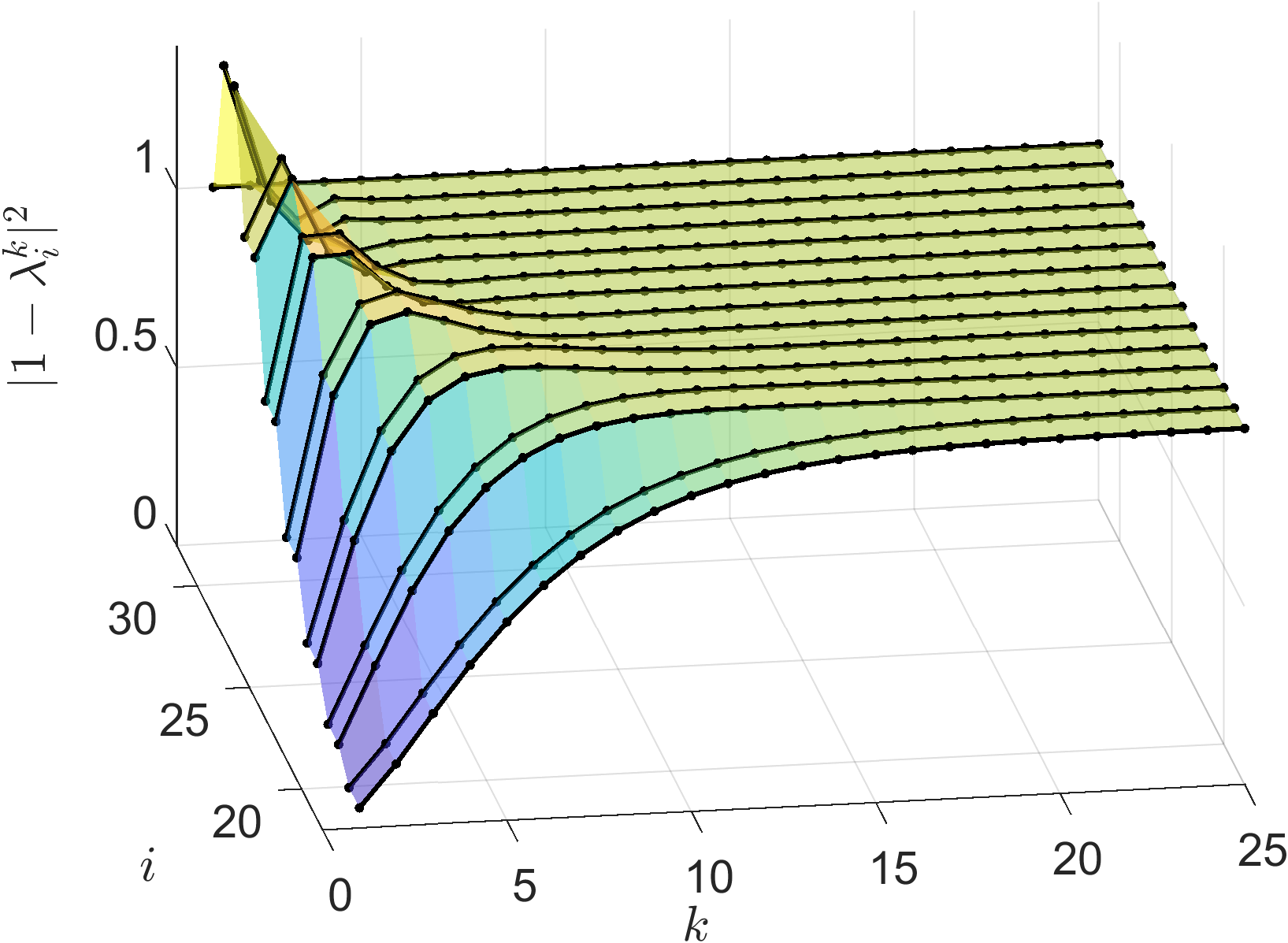}
\caption{Numerical example illustrating that $\ex\bigl(\|\bxi^k\|^2\bigr)$ can increase monotonically,
in spite of $|1-\lambda_i^k|$ not behaving monotonically with $k$. See the text for details.}
\label{fig:ExEgravity}
\end{figure}

\begin{example} \label{ex:ExEgravity} \rm
To illustrate the above-mentioned behavior and influence of $|1-\lambda_i^k|$, Figure~\ref{fig:ExEgravity} shows a numerical example using the {\sf gravity} test problem with ${\mathtt d} = 0.06$ and $n=32$.
For simplicity we assume that all $\ex(|\xi_i|^2) = 1$.
The eigenvalues $\lambda_i$ are in decreasing order and $\lambda_{32} = 0$.
While there are clearly visible bumps in the plots of $|1-\lambda_i^k|$ versus $k$
for the smaller eigenvalues
(for $i=22,\ldots,31$), the expected value $\ex(\|\bxi^k\|^2)$ increases monotonically with~$k$.
\end{example}

The analysis in this section on the expected growth of the noise error complements the analysis of
the iteration error in Section~\ref{sec:eigen}, and together they establish why we observe
semi-convergence for problems with noisy data.

\section{Conclusions}

The asymptotic convergence of the Kaczmarz method has extensively been studied in the literature (see, e.g., \cite{Pop18}).
We have mentioned the three main regimes (initial, transient, and asymptotic);
our focus here has been on a deeper understanding of the {\em initial phase} of the convergence, where the method often exhibits surprisingly fast convergence.
This speed is much appreciated by practitioners in many applications of inverse problems.
We have studied the eigenvalues of the iteration matrix $G$, which is typically nonnormal.
The spectral radius $\rho(G\vert_{\calv})$ of the iteration matrix can be very close to 1, thus explaining the very slow asymptotic convergence that is typical of the Kaczmarz method (in the absence of noise).

By considering the nullspace of $G$ we have explained the often-observed (but so far not well-understood) rapid initial convergence:
this is caused by eigencomponents in the solution corresponding to (near-)zero eigenvalues of~$G$.
For $\omega=1$ there is at least one zero eigenvalue, which also supports the practical observation that $\omega=1$ (which may not always be optimal for asymptotic convergence) is never a poor choice.
In addition, many rows of $A$ may be structurally orthogonal, especially in problems coming from CT applications, or otherwise may sometimes be arranged such that the inner products of subsequent rows are small.
This implies that there are frequently several more zero eigenvalues of $G$, further accelerating the early process.
We have also outlined several properties and advantages of choosing a small value of $\omega$.
\mh{The reason for the large plateau of the zero eigenvalue for {\sf gravity} (for $\omega \in [0.4, \, 1.6]$) is still an open problem.}

Moreover, we performed a statistical study of how the noise in the right-hand side enters the iteration vectors.
Our analysis shows that this noise component increases with the number of iterations,
and at some point it will dominate the computed solution.
The combined insight into the convergence and the influence of the noise
explains and confirms the semi-convergence that is often observed for inverse problems with noisy data,
where the reconstruction error initially decreases fast while, eventually, it grows
when the influence from the noise starts to dominate.

We have also pointed out three factors of influence
(row ordering of $A$, relaxation parameter $\omega$, and fixed point $\fixp$) on the initial and asymptotic convergence.
A row permutation may or may not improve the early stage, as this depends on the components of the solution $\fixp$ in the direction of the zero eigenvector(s) of $G$; for $\omega=1$, the first row of $A$ is such an eigenvector.
A permutation may also influence the asymptotic convergence, although it is good to bear in mind that this aspect may be less
relevant in the presence of noise and semi-convergence.

\bigskip \noindent
{\bf Acknowledgments.}
\mh{We thank two expert referees for their many useful comments.}
Some of this work has been inspired by, and is a follow-up of, the DTU--TU/e project CHARM funded by the EuroTech Postdoc Program (grant No.~754462).
Bart van Lith was funded by this program and we acknowledge his contributions.
P.~C.~Hansen has been partially supported by grant No.~25893 from the Villum Foundation.

\bibliographystyle{plain}
\bibliography{references}

\end{document}